\documentclass[journal,twocolumn]{IEEEtran}

\usepackage{cite}
\usepackage{amsmath,amssymb,amsfonts}
\usepackage{algorithmic}
\usepackage{graphicx}
\usepackage{textcomp}
\usepackage{color}
\usepackage{bm}
\usepackage{amsthm}
\newtheorem{assumption}{$Assumption$}
\usepackage{algorithm}
\newtheorem{lemma}{$Lemma$}

\newtheorem{theorem}{$Theorem$}
\newtheorem{definition}{$Definition$}
\newtheorem{corollary}{$Corollary$}

\definecolor{red1}{RGB}{192,0,0}
\definecolor{blue1}{RGB}{38,72,174}
\definecolor{purple1}{RGB}{112,48,160}
\DeclareMathOperator*{\argmin}{argmin}  
\DeclareMathOperator*{\argmax}{argmax}  

\ifCLASSINFOpdf
\else
\fi
\begin{document}
%
\title{
Accelerated Dual Averaging Methods for Decentralized Constrained Optimization
}
%
%
%


\author{Changxin Liu, \IEEEmembership{Member, IEEE},
	Yang Shi, \IEEEmembership{Fellow, IEEE},
	Huiping Li, \IEEEmembership{Senior Member, IEEE}, 
	and Wenli Du
		\thanks{This work was supported in part by the Natural Sciences and Engineering Research Council of Canada, in part by the National Natural Science Foundation of China (NSFC) under Grant 61922068, and in part by the Shaanxi Provincial Funds for Distinguished Young Scientists under Grant 2019JC-14x. \textit{(Corresponding author: Yang Shi.)} }
	\thanks{C. Liu is with the Key Laboratory of Smart Manufacturing in Energy Chemical Process, Ministry of Education, East China University of Science and Technology, Shanghai, 200237, and also with the Department of Mechanical Engineering, University of Victoria, Victoria, BC V8W 3P6, Canada (e-mail: chxliu@uvic.ca).}
	\thanks{Y. Shi is with the Department of Mechanical Engineering, University of Victoria, Victoria, BC V8W 3P6, Canada (e-mail: yshi@uvic.ca).}
	\thanks{H. Li is with the School of Marine Science and Technology, Northwestern Polytechnical University, Xi'an, 710072, China (e-mail: lihuiping@nwpu.edu.cn).}
		\thanks{W. Du is with the Key Laboratory of Smart Manufacturing in Energy Chemical Process, Ministry of Education, East China University of Science and Technology, Shanghai, 200237, China (e-mail: wldu@ecust.edu.cn).}
}

%
%

\markboth{ }
{Shell \MakeLowercase{\textit{et al.}}: Bare Demo of IEEEtran.cls for IEEE Journals}
%



\maketitle

{
\begin{abstract}

In this work, we study decentralized convex constrained optimization problems in networks. 
We focus on the dual averaging-based algorithmic framework that is well-documented to be superior in handling constraints and complex communication environments simultaneously.
Two new decentralized dual averaging (DDA) algorithms are proposed.
In the first one, a second-order dynamic average consensus protocol is tailored for DDA-type algorithms, which equips each agent with a provably more accurate estimate of the global dual variable than conventional schemes. 
We rigorously prove that the proposed algorithm attains $\mathcal{O}(1/t)$ convergence for general convex and smooth problems, for which existing DDA methods were only known to converge at $\mathcal{O}(1/\sqrt{t})$ prior to our work.
In the second one, we use the extrapolation technique to accelerate the convergence of DDA. Compared to existing accelerated algorithms, where typically two different variables are exchanged among agents at each time, the proposed algorithm only seeks consensus on local gradients. 
Then, the extrapolation is performed based on two sequences of primal variables which are determined by the accumulations of gradients at two consecutive time instants, respectively. The algorithm is proved to converge at $\mathcal{O}(1)\left(\frac{1}{t^2}+\frac{1}{t(1-\beta)^2}\right)$, where $\beta$ denotes the second largest singular value of the mixing matrix. We remark that the condition for the algorithmic parameter to guarantee convergence does not rely on the spectrum of the mixing matrix, making itself easy to satisfy in practice. 
Finally, numerical results are presented to demonstrate the efficiency of the proposed methods.

\end{abstract}}


%
\IEEEpeerreviewmaketitle

\section{Introduction}
%
%
Consider a multi-agent system consisting of $n$ agents. Each agent holds a private objective function. {They are connected via a communication network in order to collaboratively solve the following optimization problem:
\begin{equation} \label{constrained_optimization}
	\min_{x\in\mathcal{X}} \left\{  f(x):=\frac{1}{n}\sum_{i=1}^{n}f_i(x)\right\}
\end{equation}
where $f_i$ 
represents the local smooth objective function of agent $i$
and $\mathcal{X}\subseteq\mathbb{R}^m$ denotes the constraint set shared by all the agents.} This problem is referred to as decentralized optimization in the literature and finds broad applications in optimal control of cyber-physical systems, sensor networks, and machine learning, to name a few. 
For an overview of decentralized optimization and its applications, please refer to \cite{nedic2018network,yang2019survey}.



{Over the last decade, many decentralized optimization algorithms have been proposed for solving Problem \eqref{constrained_optimization}.
{\bf For unconstrained problems}, i.e., $\mathcal{X}= \mathbb{ R}^m$, the authors in} \cite{yuan2016convergence,nedic2009distributed} developed decentralized gradient descent (DGD) methods with constant step sizes, where the local search performed by individual agents is guided by local gradients and a consensus protocol. However, because each individual gradient evaluated at the global optimum is not necessarily zero, the search directions induced by consensus-seeking and local gradient may conflict with each other, making it difficult to ascertain the exact solution to the problem.
 Several efforts have been made to overcome this drawback. For example, the authors in \cite{shi2015extra} proposed the EXTRA algorithm that adds a cumulative correction term to DGD to achieve consensual optimization.  
 Alternatively, the additional gradient-tracking process based on dynamic average consensus scheme in \cite{zhu2010discrete} can be used. 
 It is shown in \cite{varagnolo2015newton,qu2017harnessing} that, for unconstrained smooth optimization, the algorithms steered by the tracked gradient exactly {converge} at an $\mathcal{O}({1}/{{t}})$ rate. 
{Based on this idea, a decentralized Nesterov gradient descent was proposed in \cite{qu2019accelerated}, where the rate of convergence is accelerated to $\mathcal{O}(1/t^{1.4-\epsilon})$ for any $\epsilon\in(0,1.4)$ at the expense of exchanging an additional variable among agents at each time instant.}
{In \cite{jakovetic2014fast}, the authors proposed an accelerated decentralized algorithm with multiple consensus rounds  at each time instant, and proved that after $t$ local iterations and $\mathcal{O}(t\log t)$ communication rounds the objective error is bounded by $\mathcal{O}(1/t^2)$}.
 By modeling Problem \eqref{constrained_optimization} as a linearly constrained optimization problem, centralized primal-dual paradigms such as the augmented Lagrangian method (ALM), the alternating direction method of multipliers (ADMM) and the dual ascent can also be used to design decentralized algorithms \cite{jakovetic2018unification,shi2014linear,uribe2021dual}. 
 {Based on the primal-dual reformulation, an accelerated primal-dual method was developed in \cite{xu2020accelerated}. The rate of convergence is improved to $\mathcal{O}(1)\left(\frac{L}{t^2}+\frac{1}{t\sqrt{\eta} t}\right)$, where $L$ denotes the smoothness parameter of each objective function and $\eta = \lambda_2(\mathcal{L})/\lambda_m(\mathcal{L})$ is the eigengap of the graph Laplacian $\mathcal{L}$. 
 Notably, the authors established a lower bound for a class of decentralized primal-dual methods, suggesting that the developed algorithm therein is optimal in terms of gradient computations. 
 The authors in \cite{scaman2017optimal} considered the Lagrangian dual formulation of the decentralized optimization problem and developed two algorithms based on dual accelerated methods. The algorithms are proved to be linearly convergent for strongly convex and smooth problems. }

{{\bf For constrained problems}, the design and convergence analysis of decentralized optimization algorithms are more challenging \cite{latafat2016new,wai2017decentralized,li2021new}. The seminal work in \cite{nedic2010constrained} is based on} the gossip protocol and projected subgradient method, where the step size was made decaying for convergence. 
{The randomized smoothing technique and multi-round consensus scheme {are} used to design a provably optimal decentralized algorithm for non-smooth convex problems in \cite{scaman2018optimal}.}
To improve the performance using a constant step size, a variant of EXTRA (PG-EXTRA) was developed in \cite{shi2015proximal}, where the constraint is modeled as a non-smooth indicator function and handled via the proximal operator. An $\mathcal{O}({1}/{t})$ rate of convergence is proved for the squared norm of the difference of consecutive iterates. 
Recently the authors in \cite{li2020decentralized} proposed an accelerated decentralized penalty method (APM),
where the constraint can be also treated as the non-smooth part of the objective.
Notably, there are some decentralized algorithms \cite{rabbat2015multi,shahrampour2017distributed,yuan2018optimal,duchi2011dual,liu2018accelerated,liu2018distributed,liang2019dual} available in the literature where the local search {mechanism} for individual agents is inspired by dual methods \cite{nemirovsky1983problem}, e.g., mirror descent and dual averaging \cite{nesterov2009primal,xiao2010dual}.
{Particularly, dual averaging is provably more efficient in exploiting sparsity than proximal gradient methods for $l_1$-regularized problems \cite{xiao2010dual}. }
For example, the authors in \cite{duchi2011dual} developed a decentralized dual averaging (DDA) algorithm where a linear model of the global objective function is gradually learned by each agent via gossip. {Compared to other types of decentralized first-order methods, DDA has the advantage in simultaneously handling constraints and complex communication environments, e.g., directed networks \cite{tsianos2012push}, deterministic time-varying networks \cite{liang2019dual}, and stochastic networks \cite{duchi2011dual,colin2016gossip}.
From a technical perspective, this is because that DDA seeks consensus on linear models of the objective function rather than on the local projected iterates as in decentralized primal methods, e.g., DGD, therefore decoupling the consensus-seeking process from nonlinear projection and facilitating the rate analysis in complex communication environments. We present a more detailed comparison in Section \ref{Subsec: Algorithms}.} 

%

 
%
 

Although decentralized dual methods in the literature have demonstrated advantages
over their primal counterparts in terms of constraint handling and analysis complexity, existing results focused on non-smooth
problems and {can have} an $\mathcal{O}({1}/{\sqrt{t}})$ rate of convergence. \emph {Considering this, a question naturally arises: If the objective functions exhibit some desired properties, e.g., smoothness, is it possible to accelerate the convergence rate of DDA beyond $\mathcal{O}({1}/{\sqrt{t}})$?}
We provide affirmative answer to this question in this work. The main results and contributions are summarized in the following:
\begin{itemize}
	\item 
{First, we develop a new DDA algorithm, where a second-order dynamic average consensus protocol is deliberately designed to assist each agent in estimating the global dual variable.
Compared to the conventional estimation scheme \cite{duchi2011dual}, the proposed method equips each agent with provably more accurate estimates.
In particular, the estimation error accumulated over time is proved to admit an upper bound constituted by the successive difference of an auxiliary variable whose update uses the mean of local dual variables.
Then a rigorous investigation into the convergence of the auxiliary variable is carried out. 
Summarizing these two relations, we {establish} conditions for algorithm parameters such that the estimation error can be fully compensated, leading to an improved rate of convergence $\mathcal{O}(1/t)$. }


\item 
Second, we propose an accelerated DDA (ADDA) algorithm.
Different from DDA, each agent employs a first-order dynamic average consensus protocol to estimate the mean of local gradients and accumulates the estimate over time to generate a local dual variable.
By solving the convex conjugate of a $1$-strongly convex function over this local dual variable, each agent produces a primal variable and uses it to construct another two sequences of primal variables in an iterative manner based on the extrapolation technique in \cite{cohen2018acceleration} and the average consensus protocol. The rate of convergence is proved to be  $\mathcal{O}(1)\left(\frac{1}{t^2}+\frac{1}{t(1-\beta)^2}\right)$, where $\beta$ denotes the second largest singular value of the mixing matrix. Notably, the condition for the algorithmic parameter to ensure convergence does not rely on the mixing matrix. Establishing such a condition that is independent on the mixing matrix offers the appealing advantage of convenient verification in practical applications.

\item 
Finally, the proposed algorithms are tested and compared with a few methods in the literature on decentralized LASSO problems characterized by synthetic and real datasets. The comparison results demonstrate the efficiency of the proposed methods.
 


\end{itemize}

\emph{Notation}: We use $\mathbb{R}$ and $\mathbb{R}^n$ to denote the set of reals and the $n$-dimensional Euclidean space, respectively.
Given a real number $a$, we denote by $\lceil a \rceil$ the ceiling function that maps $a$ to the least integer greater than or equal to $a$.
Given a vector $x\in\mathbb{ R}^n$, $\lVert x  \rVert$ denotes its $2$-norm. 
Given a matrix $P \in\mathbb{R}^{n\times n}$, its spectral radius is denoted by $\rho(P)$. Its eigenvalues and singular values are denoted by $\lambda_1(P)\geq \lambda_2(P)\geq \cdots \geq \lambda_n(P)$ and $\sigma_1(P)\geq \sigma_2(P)\geq \cdots \geq \sigma_n(P)$, respectively.

%
%
%
%
%
%

\section{Preliminaries}

\subsection {Basic Setup}

We consider the finite-sum optimization problem \eqref{constrained_optimization}, in which $\mathcal{X}$ is a convex and compact set, and $f_i$ satisfies the following assumptions for all $i=1,\dots,n$:
\begin{assumption}\label{lipschitzassumption}
	
	\begin{itemize}
		
		\item[i)] $f_i$ is continuously differentiable on $\mathcal{X}$.
		\item[ii)] $f_i$ is convex on $\mathcal{X}$, i.e., for any $x,y\in\mathcal{X}$, 
		\begin{equation}
			\label{eq:strongly-convex}
			f_i(x) - f_i(y) - \langle\nabla f_i(y), x - y \rangle \geq 0.
		\end{equation}
		\item[iii)] $\nabla f_i$ is Lipschitz continuous on $\mathcal{X}$ with a constant $L>0$, i.e., for any $x,y\in\mathcal{X}$,
		\begin{equation}
			\label{eq:Lip-origin}
			\|\nabla f_i(x) - \nabla f_i(y)\| \leq L\|x - y\|.
		\end{equation}
	\end{itemize}
\end{assumption}
A direct consequence of Assumption \ref{lipschitzassumption}(iii) is
\begin{equation}
	\label{eq:Lip-cond}
	f_i(x) - f_i(y) - \langle\nabla f_i(y), x - y \rangle \leq \frac{L}{2}\|x - y\|^2, \forall x,y\in\mathcal{X}.
\end{equation}
The above assumptions are standard in the study of decentralized algorithms for convex optimization problems. Throughout the paper, we denote by $x^*$ an optimal solution of Problem \eqref{constrained_optimization}.

%
%
%

\subsection {Communication Network}
We consider solving Problem \eqref{constrained_optimization} in a decentralized fashion, that is, a pair of agents can exchange information only if they are connected in the communication network.
To describe the network topology,
an undirected graph $\mathcal{G}=\{\mathcal{V},\mathcal{E}\}$ 
is used,
where $\mathcal{V}=\{1,\cdots,n\}$ denotes the set of $n$ agents and $\mathcal{E}\subseteq \mathcal{V}\times\mathcal {V}$ represents the set of bidirectional channels, i.e., $(i,j)\in \mathcal{E}$ indicates that nodes $i$ and $j$ can send information to each other. 
Agent $j$ is said to be a neighbor of $i$ if there exists a link between them, 
and the set of $i$'s neighbors is denoted by $ \mathcal{N}_i=\{j\in \mathcal{V}|(j,i)\in \mathcal{E} \}$.
For every pair $(i,j)\in\mathcal{E}$, a positive weight $p_{ij}>0$ is assigned to $i$ and $j$ to weigh the information received from each other. Otherwise $p_{ij}=0$ is considered. 
For the convergence of the algorithm, we make the following assumption for $P:=[p_{ij}]\in [0,1]^{n\times n}$.

\begin{assumption}\label{graphconnected}
	\begin{itemize}
		\item[i)] $P\mathbf{1} = \mathbf{1}$ and $\mathbf{1}^{\mathrm{T}}P = \mathbf{1}^{\mathrm{T}}$, where $\mathbf{1}$ denotes the all-one vector of dimension $n$.
		\item[ii)] $P$ has a strictly positive diagonal. i.e., $p_{ii}>0$.
	\end{itemize}
\end{assumption}


Assumption \ref{graphconnected} implies that $\sigma_2(P)<1$ \cite{liu2018distributed}.
Given a connected network, the constant edge weights and the Metropolis-Hastings algorithm \cite{xiao2004fast} can be used to construct a weight matrix $P$ fulfilling Assumption \ref{graphconnected}.




\subsection{Centralized Dual Averaging}

Our algorithms are based on the dual averaging methods \cite{nesterov2009primal}. 
Before introducing them, we state the following definition.
\begin{definition}\label{prox_function}
	A differentiable function $\psi$ is strongly convex with modulus $\mu> 0$ on $\mathcal{X}$, if
	\begin{equation*}
		\psi(x)- \psi(y)-\langle\nabla \psi(y), x-y \rangle\geq\frac{\mu}{2}\lVert x-y\rVert^2, \forall x,y\in\mathcal{X}.
	\end{equation*}
\end{definition}
Let $d$ be a strongly convex  and differentiable function with modulus $1$ on $\mathcal{X}$ such that
\begin{equation}\label{eq:initial-cond}
	x^{(0)} =\argmin_{x\in\mathcal{X}} d(x) \quad  \mathrm{and} \quad d(x^{(0)})=0.
\end{equation}
{To meet the condition in \eqref{eq:initial-cond} for any $x^{(0)}\in\mathcal{X}$, one can choose
\begin{equation*}
	d(x) = \tilde{d}(x)-\tilde{d}(x^{(0)}) - \langle \nabla \tilde{d}(x^{(0)}), x-x^{(0)} \rangle,
\end{equation*} 
where $\tilde{d}$ is any strongly convex function with modulus $1$, e.g., $\tilde{d}(x)=\lVert x \rVert^2/2$.}
The convex conjugate of $d$ is defined as
\begin{equation*}
	\nabla d^*(\cdot) =\argmax_{x\in\mathcal{X}} \left\{ \left\langle \cdot, x\right\rangle -d(x) \right\}.
\end{equation*}
As a corollary of Danskin's Theorem, we have the following result \cite{cohen2018acceleration}.
\begin{lemma} \label{gamma_continuity}
	For all $x,y\in\mathbb{R}^m$, we have
	\begin{equation}\label{projection_lipschitz}
		\begin{split}
			&\left\lVert \nabla d^*(x) -\nabla d^*(y) \right\rVert  \leq  \lVert x-y\rVert.
		\end{split}
	\end{equation}
\end{lemma}

{\bf Dual averaging.} 
The dual averaging method can be applied to solving Problem \eqref{constrained_optimization} in a centralized manner. Starting with $x^{(0)}$, it generates a sequence of variables $\{x^{(t)}\}_{t\geq 0}$ iteratively according to

\begin{equation}\label{cda}
	{x}^{(t)}= \nabla d^*\left(  - a_tz^{(t)}  \right) 
\end{equation}
where 
\begin{equation}\label{c_dual}
	z^{(t)} = \sum_{\tau = 0}^{t-1}\nabla f(x^{(\tau)})
\end{equation}
and $\{a_t\}_{t\geq 0}$ is a sequence of positive 
parameters that determines the rate of convergence.
Let $\tilde{x}^{(t)} = t^{-1}\sum_{\tau=0}^{t-1}{x}_i^{(\tau)}$.
It is proved in \cite{nesterov2009primal} 
that $f(\tilde{x}^{(t)})-f(x^*) \leq \mathcal{O}({1}/{\sqrt{t}})$ 
when $a_t=\Theta(1/\sqrt{t})$, that is, with order exactly $1/\sqrt{t}$. When the objective function is convex and  smooth, a
constant $a_t=a$ can be used to achieve an ergodic
$\mathcal{O}({1}/{t})$ rate of convergence in terms of objective error \cite{lu2018relatively}.

{{\bf Accelerated dual averaging.} To speed up the rate of convergence, an accelerated dual averaging algorithm is developed in \cite{cohen2018acceleration}. In particular, the variables are updated according to
\begin{subequations}\label{ADA_iteration}
	\begin{align}
		u^{(t)}&=\frac{A_{t-1}}{A_{t}}v^{(t-1)}+\frac{a_{t}}{A_{t}}{w}^{(t-1)} \\
		v^{(t)}&=\frac{A_{t-1}}{A_{t}}v^{(t-1)}+\frac{a_{t}}{A_{t}}	{w}^{(t)},
	\end{align}
\end{subequations}
where {$a_t:=a(t+1)$} for some $a>0$, $A_t=\sum_{\tau=1}^{t}a_\tau$ and
\begin{equation} \label{projection_ADA}
	{w}^{(t)}=\nabla d^*\left(  - \sum_{\tau=1}^{t}a_\tau \nabla f(u^{(\tau)})  \right).
\end{equation}}
Note that $t\geq 2$ is considered for the above iteration, and the variables are initialized with $u^{(1)}=w^{(0)} = x^{(0)}$, $v^{(1)}=w^{(1)} $.
For convex and smooth objective functions, it is proved that $f(v^{(t)})-f(x^*)\leq \mathcal{O}(1/t^2)$ \cite{cohen2018acceleration}.

\section{Algorithms and Convergence Results}
In this section, we develop two new DDA algorithms that are provably more efficient than existing DDA-type algorithms.

\subsection {Decentralized Dual Averaging}\label{Subsec: Algorithms}

To solve Problem \eqref{constrained_optimization} in a decentralized manner, we propose a novel DDA algorithm.
In particular, we employ the following dynamic average consensus protocol to estimate $z^{(t)}$ in \eqref{c_dual}: 
\begin{subequations}\label{2nd_order_consensus}
	\begin{align}
			z_{i}^{(t)}&=\sum_{j=1}^np_{ij}\left( z_{j}^{(t-1)}+s_{j}^{(t-1)}  \right), \label{def:z}\\
	s_{i}^{(t)}&=\sum_{j=1}^np_{ij} s_{j}^{(t-1)}+\nabla f_i(x^{(t)}_i)-\nabla f_i(x^{(t-1)}_i)  , \label{def:s}
	\end{align}
\end{subequations}
where $z_{i}^{(t)}$ is a local estimate of $z^{(t)}$ generated by agent $i$, $s_i^{(t)}$ is a proxy of $\frac{1}{n}\sum_{i=1}^{n}\nabla f_i(x_i^{(t)})$ which aims to reduce the consensus error among variables $\{z_i^{(t)}:i=1,\cdots, n\}_{t\geq 0}$.
Using it, each agent $i$ performs a local computation step to update its estimate of $x^{(t)}$:
\begin{equation}\label{primal_smooth}
{x}_{i}^{(t)}= \nabla d^*\left(  -az_i^{(t)}  \right).
\end{equation}
The overall algorithm is summarized in Algorithm \ref{FDDA}.

\begin{algorithm}[tb]
	\caption{Decentralized Dual Averaging}
		\label{FDDA}
	\begin{algorithmic}
		\STATE {\bfseries Input:} $a>0$, $x^{(0)}\in\ \mathcal{X}$ and a strongly convex function $d$ with modulus $1$ such that \eqref{eq:initial-cond} holds
		\STATE {\bfseries Initialize:} $x_i^{(0)} = x^{(0)}$, $z_i^{(0)} = 0$, and $s_i^{(0)} = \nabla f_i(x^{(0)})$ for all $i = 1, \dots, n$ 
		\FOR{$t=1,2,\cdots$}
		\STATE \emph{In parallel (task for agent $i$, $i = 1,\dots,n$)}
		\STATE collect $z_{j}^{(t-1)}$ and $s_{j}^{(t-1)}$ from all agents $j\in\mathcal{N}_{i}$
		\STATE update $z_{i}^{(t)}$ and $s_{i}^{(t)}$ by \eqref{2nd_order_consensus}
		\STATE compute $x_i^{(t)}$ by \eqref{primal_smooth}
		\STATE broadcast $z_{i}^{(t)}$ and $s_{i}^{(t)}$ to all agents $j\in\mathcal{N}_{i}$
		\ENDFOR
	\end{algorithmic}
\end{algorithm}

Before proceeding, we make the following remarks on Algorithm \ref{FDDA}.


{\textbf{i) Subproblem solvability.}  Similar to centralized dual averaging methods, we assume the subproblem in \eqref{primal_smooth} can be solved easily. For general problems, we can choose $d(x)=\lVert x-x^{(0)} \rVert^2/2$ such that the subproblem \eqref{primal_smooth} reduces to computing the projection of variables onto $\mathcal{X}$. If $\mathcal{X}$ is simple enough, e.g., the simplex or $l_1$-norm ball, a closed-form solution exists.

\textbf{ii) Comparison with existing DDA algorithms.}
In existing DDA algorithms \cite{duchi2011dual,liang2019dual,liu2018distributed},  each agent estimates $z^{(t)}$ in the following way
	\begin{equation}\label{d_dual}
		z_{i}^{(t)}=\sum_{j=1}^np_{ij}z_{j}^{(t-1)}+\nabla f_i(x_{i}^{(t)}).
	\end{equation}
For this scheme, it is proved that the consensus error among variables $\{z_i^{(t)}:i=1,\cdots, n\}_{t\geq 0}$ admits a constant upper bound \cite{duchi2011dual}, which necessitates the use of a monotonically decreasing sequence $\{a_t\}_{t\geq 0}$ for convergence. However, decreasing $a_t$ slows down the convergence significantly; the rate of convergence in \cite{duchi2011dual,liang2019dual} is reported to be $\mathcal{O}(1/\sqrt{t})$. To speed up the convergence, we develop the consensus protocol in \eqref{2nd_order_consensus}, 
which is inspired by the high-order consensus scheme in \cite{zhu2010discrete}.
Thanks to it, we are able to prove that the deviation among variables $\{z_i^{(t)}:i=1,\cdots, n\}_{t\geq 0}$ asymptotically vanishes as time evolves. 
Therefore, the parameter in \eqref{primal_smooth} can be set constant, i.e, $a_t=a>0$, which is key to obtaining the improved rates.}


	
\textbf{iii) Comparison with DGD algorithms.}
In existing DGD, a so-called gradient-tracking process similar to \eqref{2nd_order_consensus} is usually observed:
\begin{equation}\label{gradient-tracking}
	\begin{split}
		x_{i}^{(t)}&=\sum_{j=1}^np_{ij}\left( x_{j}^{(t-1)}+as_{j}^{(t-1)}  \right),\\
		s_{i}^{(t)}&=\sum_{j=1}^np_{ij} s_{j}^{(t-1)}+\nabla f_i(x^{(t)}_i)-\nabla f_i(x^{(t-1)}_i)  ,
	\end{split}
\end{equation}
	where $a$ represents the step size. 
The proposed scheme \eqref{2nd_order_consensus} differs from \eqref{gradient-tracking} in step \eqref{def:z}. With such a deliberate design and another local dual averaging step in \eqref{primal_smooth}, Algorithm \ref{FDDA} solves \emph{constrained} problems with convergence rate guarantee.
	To compare DDA with existing algorithms applicable to solving Problem \eqref{constrained_optimization}, we recall the PG-EXTRA algorithm \cite{shi2015proximal}:
	\begin{equation}\label{pg-extra}
		\begin{split}
			\hat{x}_{i}^{(t+1)}=&\sum_{j=1}^{n}p_{ij}{x}_{j}^{(t)}+\hat{x}_{i}^{(t)}-\sum_{j=1}^{n}\tilde{p}_{ij}{x}_{j}^{(t-1)}\\
			&-a\left( \nabla f_i(x_{i}^{(t)})-\nabla f_i(x_{i}^{(t-1)})\right)\\
			{x}_{i}^{(t+1)} =& \argmin_{x\in\mathcal{X}} \left\lVert x-\hat{x}_{i}^{(t+1)}\right\rVert^2
		\end{split}
	\end{equation}
	where $a$ represents the step size and 
	$\tilde{p}_{ij}^{(t)}$ denotes the $(i,j)$-th entry of $\tilde{ P}={(P+I)}/{2}$. {Notably, PG-EXTRA seeks consensus among variables $\{{  x}_{i}^{(t)}:i=1,\cdots,n\}$ at time $t+1$ that are obtained via a projection operator at time $t$.
	In contrast, DDA manages to agree on $\{{  z}_{i}^{(t)}:i=1,\cdots,n\}$, which essentially decouples the consensus-seeking procedure from projection. 
	After using the smoothness assumption in \eqref{eq:Lip-origin} to bound $\lVert \nabla f_i(x^{(t)}_i)-\nabla f_i(x^{(t-1)}_i) \rVert$ in \eqref{def:s}, the iteration in \eqref{2nd_order_consensus} can be kept almost linear, which greatly facilitates the rate analysis; see the proof of Lemma 
	\ref{consensus_error_lemma} for more details.
	
}

%
%


Next, we present the convergence result of Algorithm \ref{FDDA}.
Inspired by \cite{duchi2011dual}, we first 
establish the convergence property of an auxiliary sequence $\{{y}^{(t)}\}_{t\geq 0}$, which is instrumental in proving the convergence of $\{x_i^{(t)}: i =1 ,\cdots,n\}_{t\geq 0}$. In particular, the update of ${y}^{(t)}$ obeys
\begin{equation}\label{auxiliary_sequence}
	{y}^{(t)}= \nabla d^*\left(  - a\overline{z}^{(t)}  \right),
\end{equation}
where the initial vector ${y}^{(0)}=x^{(0)}$ and $\overline{z}^{(t)}=\frac{1}{n}\sum_{i=1}^{n}z_i^{(t)}$.
To proceed, we introduce the following $2\times 2$ matrix:
	\begin{equation}\label{def:M}
	{\bf M}=\begin{bmatrix}
		\beta & \beta \\
		aL(\beta+1) & \beta(aL+1)
	\end{bmatrix}
\end{equation}
where $\beta =\sigma_2(P)$,
and let $\rho({\bf M})$ be the spectral radius of ${\bf M}$.

\begin{theorem}\label{inexact_thm}
	Suppose that Assumptions \ref{lipschitzassumption} and \ref{graphconnected} are satisfied. If 
the constant $a$ in Algorithm \ref{FDDA} satisfies 
		\begin{equation}\label{1st_condition}
\frac{1}{a}>2L\cdot\max \left\{\frac{\beta}{(1-\beta)^2}, 1 +\frac{8 }{9\left(1-(\rho({\bf M}))^2\right)}\right\},
	\end{equation}	
then, for all $t\geq 1$, it holds that
	\begin{equation}\label{FDDA_rate}
		\begin{split}
			f(\tilde{y}^{(t)})  -f(x^*)\leq \frac{C}{at},
		\end{split}
	\end{equation}
	where $\tilde{y}^{(t)}=t^{-1}\sum_{\tau=1}^{t}{y}^{(\tau)}$ with ${y}^{(\tau)}$ defined in \eqref{auxiliary_sequence}, 
	\begin{equation*}
		C:=d(x^*) +\frac{8a\pi^2 }{9nL\left(1-(\rho({\bf M}))^2\right)},
	\end{equation*}
and 
\begin{equation}\label{gradient_variance_t0}
	\pi^2 =  {\sum_{i=1}^n\left\|\nabla f_i(x^{(0)}) - \frac{1}{n}\sum_{j=1}^{n}\nabla f_j (x^{(0)})\right\|^2}.
\end{equation}
In addition, for all $t\geq 1$ and $i=1,\cdots, n$, we have
	\begin{equation}	    \label{eq:x-y-dist}
		\lVert \tilde{x}_i^{(t)} - \tilde{y}_i^{(t)}  \rVert^2  \leq \frac{D}{t}
	\end{equation}
where $\tilde{x}_i^{(t)}=t^{-1}\sum_{\tau=1}^{t}{x}_i^{(\tau)}$,
\begin{equation*}
	D:=\frac{8nC}{9\gamma\left(1-\rho({\bf M})\right)^2} + \frac{8\pi^2 }{9L^2\left(1-(\rho({\bf M}))^2\right)} ,
\end{equation*}
and 	\begin{equation}\label{def:gamma}
	\gamma := \frac{1}{2} -aL-\frac{8aL}{9\left(1-(\rho({\bf M}))^2\right)}.
\end{equation}
\end{theorem}  
\begin{proof}
	The proof is postponed to Appendix A.
\end{proof}
{
To obtain a more explicit version of \eqref{1st_condition}, we identify the eigenvalues of ${\bf M}$ as $\lambda_1 = (\xi_1+\xi_2)/2$ and $\lambda_2 = (\xi_1-\xi_2)/2$, where
\begin{equation}\label{eig}
	\xi_1 = \beta(2+aL)>0, \quad \xi_2= \sqrt{a^2\beta^2L^2+4aL\beta(\beta+1)}>0.
\end{equation}
Thus, we have $\lvert \lambda_1 \rvert>\lvert \lambda_2 \rvert$ and $\rho({\bf M}) =\lambda_1>0$. By routine calculation, we can verify that $\rho({\bf M})$ and $\nu(a):=\frac{8}{9(1-(\rho({\bf M}))^2)}$ monotonically increase with $a$. Due to 
\begin{equation*}
	\nu(\frac{1}{2L}) < \frac{1}{\left(1-\left(\frac{2.5\beta +\sqrt{2.25\beta^2+2\beta}}{2}\right)^2\right)},
\end{equation*}
we have that as long as $a$ satisfies
\begin{small}
\begin{equation}\label{explicit_parameter}
	\frac{1}{a}>2L\cdot\max \left\{\frac{\beta}{(1-\beta)^2}, 1 +\frac{1}{\left(1-\left(\frac{2.5\beta +\sqrt{2.25\beta^2+2\beta}}{2}\right)^2\right)}\right\},
\end{equation}	
\end{small}
then $a$ also satisfies \eqref{1st_condition}.
Based on \eqref{explicit_parameter}, we have
\begin{equation*}
	a =	\Theta\left(\frac{(1-\beta)^2}{L}\right),
\end{equation*}
{whose size} is comparable to the DGD algorithms \cite{nedic2017achieving,qu2017harnessing,xu2015augmented} in the literature.}

Next, we consider an unconstrained version of Problem \eqref{constrained_optimization}, i.e., $\mathcal{X}=\mathbb{ R}^m$, where the rate of convergence is stated for $f(\tilde{x}_{i}^{(t)})-f(x^*)$.

\begin{corollary}\label{unconstrained corollary}
	Suppose 
	the premise of Theorem \ref{inexact_thm} holds.
If	$\mathcal{X}=\mathbb{ R}^m$ in \eqref{constrained_optimization} and $d(x)=\lVert x\rVert^2/2$ in \eqref{primal_smooth}, and
			\begin{equation}\label{step size_FDDA_unconstrained}
		\frac{1}{a}>2L\cdot\max \left\{\frac{\beta}{(1-\beta)^2}, 1 +\frac{8 }{3\left(1-(\rho({\bf M}))^2\right)}\right\},
	\end{equation}	
	then 
	\begin{equation}\label{rate_FDDA_unconstrained}
		\begin{split}
			f(\tilde{x}_{i}^{(t)})  -f(x^*)\leq  
			\frac{1}{t}\left( \frac{n}{2a}\lVert x^* \rVert^2+ \frac{8\pi^2 }{3L\left(1-(\rho({\bf M}))^2\right)} \right)
		\end{split}
	\end{equation}
where $\tilde{x}_i^{(t)}=t^{-1}\sum_{\tau=1}^{t}{x}_i^{(\tau)}$ and $\pi^2$ is defined in \eqref{gradient_variance_t0}.
\end{corollary}
\begin{proof}
		The proof is given in Appendix B.
\end{proof}

\subsection{Accelerated Decentralized Dual Averaging}

To further speed up the convergence, we develop a decentralized variant of the accelerated dual averaging method in \eqref{ADA_iteration} and \eqref{projection_ADA}. Different from Algorithm \ref{FDDA}, {we consider building consensus among variables $\{v_i^{(t)},i=1,\cdots,n\}$ and propose the following iteration rule:
\begin{subequations}\label{ADDA_iteration}
	\begin{align}
	u^{(t)}_{i}&=\frac{A_{t-1}}{A_{t}}\sum_{j=1}^np_{ij}v^{(t-1)}_{j}+\frac{a_{t}}{A_{t}}{w}_{i}^{(t-1)}  \label{update_u}\\
	v^{(t)}_{i}&=\frac{A_{t-1}}{A_{t}}\sum_{j=1}^np_{ij}v^{(t-1)}_{j}+\frac{a_{t}}{A_{t}}	{w}_{i}^{(t)}, \label{update_v}
	\end{align}
\end{subequations}
where 
\begin{equation} \label{projection_ADDA}
	{w}_{i}^{(t)}= \nabla d^*\left(  - \sum_{\tau=1}^{t}a_\tau q_{i}^{(\tau)}  \right),
\end{equation}
and 
\begin{equation}\label{consensus_ADDA}
	q_{i}^{(t)}=\sum_{j=1}^np_{ij}q_{j}^{(t-1)}+\nabla f_i(u^{(t)}_{i})-\nabla f_i(u^{(t-1)}_{i}).
\end{equation}
}
The overall algorithm is summarized in Algorithm \ref{ADDA}. 
It is worth to mention that agents in Algorithm \ref{ADDA} consume the same communication resources as Algorithm \ref{FDDA} to achieve acceleration.

\begin{algorithm}[tb]
	\caption{Accelerated Decentralized Dual Averaging}
	\label{ADDA}
	\begin{algorithmic}
		\STATE {\bfseries Input:} $a>0$, $x^{(0)}\in\mathcal{X}$ and a strongly convex function $d$ with modulus $1$ such that \eqref{eq:initial-cond} holds
		\STATE {\bfseries Initialize:}  $A_1 =a_1= 2a$, $u_i^{(1)}=w_i^{(0)} = x^{(0)}$, $q_i^{(1)} = \nabla f_i(x^{(0)})$, and $v_i^{(1)}=w_i^{(1)} $ for all $i = 1, \dots, n$ 
		\FOR{$t=2,3,\cdots$}
		\STATE set $a_{t}=a_{t-1}+a$ and $A_{t}=A_{t-1}+a_{t}$
		\STATE \emph{In parallel (task for agent $i$, $i = 1,\dots,n$)}
		\STATE collect $v_{j}^{(t-1)}$ and $q_{j}^{(t-1)}$ from all agents $j\in\mathcal{N}_{i}$
		\STATE update $u_i^{(t)}$ by \eqref{update_u}
		\STATE update $q_{i}^{(t)}$ by \eqref{consensus_ADDA}
		\STATE compute $w_i^{(t)}$ by \eqref{projection_ADDA}
						\STATE update $v_i^{(t)}$ by \eqref{update_v}
		\STATE broadcast $v_{i}^{(t)}$ and $q_{i}^{(t)}$ to all agents $j\in\mathcal{N}_{i}$
		\ENDFOR
	\end{algorithmic}
\end{algorithm}


\begin{assumption}\label{boundedness}
	For the problem in \eqref{constrained_optimization}, the constraint set $\mathcal{X}$ is bounded with the following diameter:
	\begin{equation*}
		G = \max_{x,y\in\mathcal{X}}\lVert x-y \rVert.
	\end{equation*}
\end{assumption}

\begin{theorem}\label{ADDA_thm}
	For Algorithm \ref{ADDA}, 
	if Assumptions \ref{lipschitzassumption}, \ref{graphconnected}, and \ref{boundedness} are satisfied, and 
	\begin{equation}\label{ADDA_condition}
		a\leq\frac{1}{6L},
	\end{equation}
 then, for all $t\geq 1$, it holds that
	\begin{equation}\label{ADDA_rate}
		\begin{split}
			f(\overline{  v}^{(t)})-f(x^*)	\leq \frac{d(x^*)}{A_t}+\frac{t}{A_t}\left(\frac{2G(LC_p+C_g)}{\sqrt{n}}+\frac{6LC_p^2}{n}\right),
		\end{split}
	\end{equation}
	where $$C_p:=\lceil\frac{3}{1-\beta}\rceil\sqrt{n}G$$ and $$C_g:=2L\lceil\frac{3}{1-\beta}\rceil \frac{ \sqrt{n}G+C_p}{1-\beta}.$$ In addition, for all $t\geq 1$ and $i=1,\cdots, n$, we have
	\begin{equation}\label{consensus_error_ADDA}
	\left\lVert {  v}_i^{(t)}-\overline{  v}^{(t)} \right\rVert^2 \leq\frac{2aC_p}{A_{t}}.
\end{equation}

\end{theorem}

\begin{proof}
	The proof is postponed to Appendix C.
\end{proof}

For Algorithm \ref{ADDA} and Theorem \ref{ADDA_thm}, the following remarks are in order.


\textbf{i) Comparison with existing accelerated algorithms.} Accelerated methods for decentralized constrained optimization are rarely reported in the literature. Recently, the authors in \cite{li2020decentralized} developed the APM algorithm, where the iteration rule reads
\begin{subequations}
	\begin{align}
		y_{i}^{(t)} =&{x}_{i}^{(t)}+\frac{\theta_t(1-\theta_{t-1})}{\theta_{t-1}}\left({x}_{i}^{(t)}-{x}_{i}^{(t-1)}\right) \\
		{s}_{i}^{(t)}=& \nabla f_i(y_{i}^{(t)})+\frac{\beta_0}{\theta_t}\sum_{i=1}^{n}p_{ij}\left( y_{i}^{(t)}-y_{j}^{(t)}\right) \label{APM_gradient} \\
		{x}_{i}^{(t+1)} =&\arg\min_{{x}\in\mathcal{X}}\left\lVert x-y_{i}^{(t)}+\frac{s_{i}^{(t)}}{L+\beta_0/\theta_t}\right\rVert^2 \label{APM_proj}
	\end{align}
\end{subequations}
where $\beta_0={L}/{\sqrt{1-\lambda_2(P)}}$ and $\theta_k$ is a decreasing parameter satisfying $\theta_k={\theta_{k-1}}/{(1+\theta_{k-1})}$ with $\theta_0=1$. Letting $\hat{s}_{i}^{(t)}=\theta_ks_{i}^{(t)}$, we can equivalently rewrite \eqref{APM_gradient} and \eqref{APM_proj} as
\begin{equation*}
	\begin{split}
		\hat{s}_{i}^{(t)}=& {\theta_t}\nabla f_i(y_{i}^{(t)})+{\beta_0}\sum_{i=1}^{n}p_{ij}\left( y_{i}^{(t)}-y_{j}^{(t)}\right)  \\
		x_{i}^{(t+1)} =&\arg\min_{{x}\in\mathcal{X}}\left\lVert x-y_{i}^{(t)}+\frac{\hat{s}_{i}^{(t)}}{L\theta_t+\beta_0}\right\rVert^2,
	\end{split}
\end{equation*}
from which we can see that new gradients are assigned with decreasing weights, whereas increasing weights are used for ADDA in \eqref{projection_ADDA}. 
{The reason for such different choices of parameters may be two-fold. First, parameter choices in (centralized) primal gradient descent and dual averaging methods are intrinsically different. Second, APM gradually increases the penalty parameter $1/\theta_t$ in order to enforce consensus, which essentially dilutes the weight for gradients, as shown above.}
We will show in simulation that  decreasing weights over time slows down convergence.
There are also a few other accelerated decentralized methods such as \cite{qu2019accelerated,xu2020accelerated}, however they do not apply to constrained problems.

	{\textbf{ii) Discussion about optimality.}
For ADDA, the rate of convergence is proved to be
	\begin{equation*}
		\mathcal{O}(1)\left(  \frac{1}{t^2}+ \frac{1}{t(1-\beta)^2} \right).
	\end{equation*}
	In light of the lower bound in \cite{xu2020accelerated}, it is not optimal in terms of the dependence on $\beta$.
	In particular, the dominant term of the error in $\mathcal{O}(1/(t(1-\beta)^2))$ becomes larger as $
	\beta$ grows, i.e., the network becomes more sparsely connected.
	This is mainly because we consider a one-consensus-one-gradient update in the algorithm. However, extending the algorithm in \cite{xu2020accelerated} to handle constraints may require further investigation.
	In the simulation section, we demonstrate the superiority of ADDA over existing decentralized constrained optimization algorithms.}

\section{Simulation}
{In this section, we verify the proposed methods by applying them to solve the following constrained LASSO problems:
\begin{equation*}
	\min_{x\in\mathbb{R}^m}  \left\{f(x)=\frac{1}{2n}\sum_{i=1}^{n}\lVert M_ix - c_i \rVert^2 \right\}, \quad \mathrm{s.t.} \,\, \lVert x \rVert_1 \leq R
\end{equation*} 
where $M_i\in \mathbb{R}^{p_i\times m}$, $c_i\in\mathbb{R}^{p_i} $,
and $R$ is a constant parameter that defines the constraint. 
In the simulation, each agent $i$ has access to a local data tuple $(y_i,A_i)$ and $R$.
Two different problem instances characterized by both real and synthetic datasets are considered. 

%


\subsection{Case I: Real Dataset}
In this setting, we use sparco7 \cite{van2007sparco, wai2017decentralized} to define the LASSO problem, and consider a cycle graph and a complete graph of $n=50$ nodes.
The corresponding weight matrix $P$ is determined by following the Metropolis-Hastings rule \cite{xiao2004fast}.
Each local measurement matrix $M_i\in\mathbb{R}^{12\times 2560}$, and the local  corrupted measurement $c_i\in\mathbb{ R}^{12}$. The constraint parameter is set as $R = 1.1\cdot\lVert x_{{g}}  \rVert_1$, where $x_{{g}}$ with $\lVert x_g \rVert_0= 20$ denotes the unknown variable to be recovered via solving LASSO. In this case, the simulation experiments were performed using MATLAB R2020b.

For comparison, the PG-EXTRA method in \cite{shi2015proximal} and the APM method in \cite{liu2018accelerated}
are simulated. 
For their algorithmic parameters, the step size for PG-EXTRA is set as $10^{-4}$, and the parameters for APM are set as $L=250$ and $\beta_0={L}/{\sqrt{1-\lambda_2(P)}}$. 
For DDA and ADDA in this work, we use $a=5\cdot10^{-4}$ and $a_t=(t+1)\cdot10^{-4}$, respectively, and $\lVert x \rVert^2/2$ as the prox-function.
The projection onto an $l_1$ ball is carried out via the algorithm in \cite{condat2016fast}.
All the methods are initialized with
$x_{i}^{(0)}=0, \forall i \in\mathcal{V}$.

The performance of four algorithms are displayed in Figs. \ref{fig:obj} and \ref{fig:mse}. In Fig. \ref{fig:obj}, the performance is evaluated in terms of the objective error  $f(\frac{1}{n}\sum_{i=1}^{n}x_i^{(t)})-f(x^*)$, where $x^*$ is identified using CVX \cite{grant2014cvx}. It demonstrates that the DDA method outperforms other methods when the graph is a cycle. As the graph becomes denser, i.e, complete graph, the convergence of all algorithms becomes faster. Among them, the ADDA method demonstrates the most significant improvement. This is in line with Theorem \ref{ADDA_thm}, where the network connectivity impacts the convergence error in $\mathcal{O}(1/t)$ as opposed to $\mathcal{O}(1/t^2)$. In Fig. \ref{fig:mse}, we compare the trajectories of consensus error, i.e., $\sqrt{\sum_{i=1}^{n}\lVert x_i^{(t)} - n^{-1}\sum_{j=1}^nx_j^{(t)}\rVert^2}$, by all methods. When the graph is a cycle, the APM and PG-EXTRA have smaller consensus error than the developed methods, mainly because they build consensus directly among variables $\{x_i^{(t)}:i=1,\cdots, n\}_{t\geq 0}$. When the graph is complete, the consensus error by the proposed DDA method vanishes because of the conservation property established in Lemma \ref{conservation_property} and the complete graph structure.}

%
%

{\subsection{Case II: Synthetic Dataset}
For the synthetic dataset, the parameters are set as $n=8$, $m=30000$, $p_i=2000, \forall i\in\mathcal{V}$, and the data is generated in the following way.
First, each local measurement matrix $M_i$ is randomly generated where each entry follows the normal distribution $\mathcal{N}(0,1)$. Next, each entry of the sparse vector $x_{{g}}$ to be recovered via LASSO is randomly generated from the normal distribution $\mathcal{N}(0,1)$ with $\lVert x_{{g}} \rVert_0=1500$. Then the corrupted measurement $c_i$ is produced based on
\begin{equation*}
	c_i=M_ix_g+b_i
\end{equation*}
where $b_i$ represents the Gaussian noise with zero mean and variance $0.01$. The constraint parameter is set as $R= 1.1\cdot\lVert x_g \rVert_1$. 
For this setting, we employed the message passing interface (MPI) in Python 3.7.3 to simulate a network of $8$ nodes, where each node $i$ is connected to a subset of nodes $\{1+ i \, \mathrm{mod} \,8, 1+ (i+3) \, \mathrm{mod} \, 8, 1+ (i+6) \, \mathrm{mod} \, 8\}$. For comparison, the proposed methods are compared to their centralized counterparts. The parameters for dual averaging and accelerated dual averaging are set as $a=1/(3\cdot10^5)$ and $a_t = a(t+1)$, respectively. Similarly, the function $\lVert x \rVert^2/2$ is used as a prox-function, and the algorithms are initialized with
$x_{i}^{(0)}=0, \forall i \in\mathcal{V}$.

The performance of the developed algorithms and their centralized counterparts is illustrated in Fig. \ref{fig:robj}. In particular, the performance is evaluated in terms of objective function value versus computing time.
It demonstrates that the proposed methods outperform the corresponding centralized algorithms in the sense that the decentralized algorithms consume less computing time to reach the same degree of accuracy than their centralized counterparts. }

\begin{figure}
	\begin{center}
		\includegraphics[width=8.8cm]{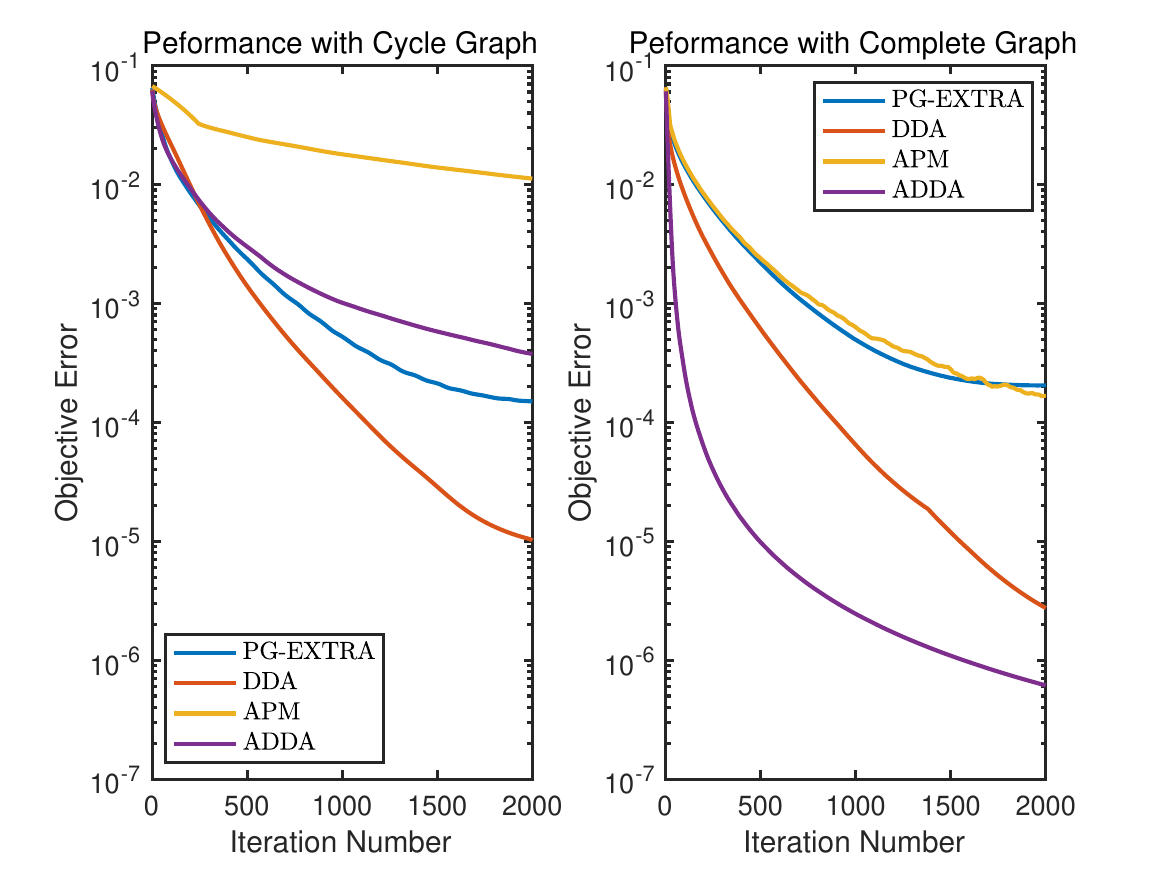}    
		\caption {Comparison of objective error in Case I. }
		
		\label{fig:obj}
	\end{center}
\end{figure}

\begin{figure}
	\begin{center}
		\includegraphics[width=8.8cm]{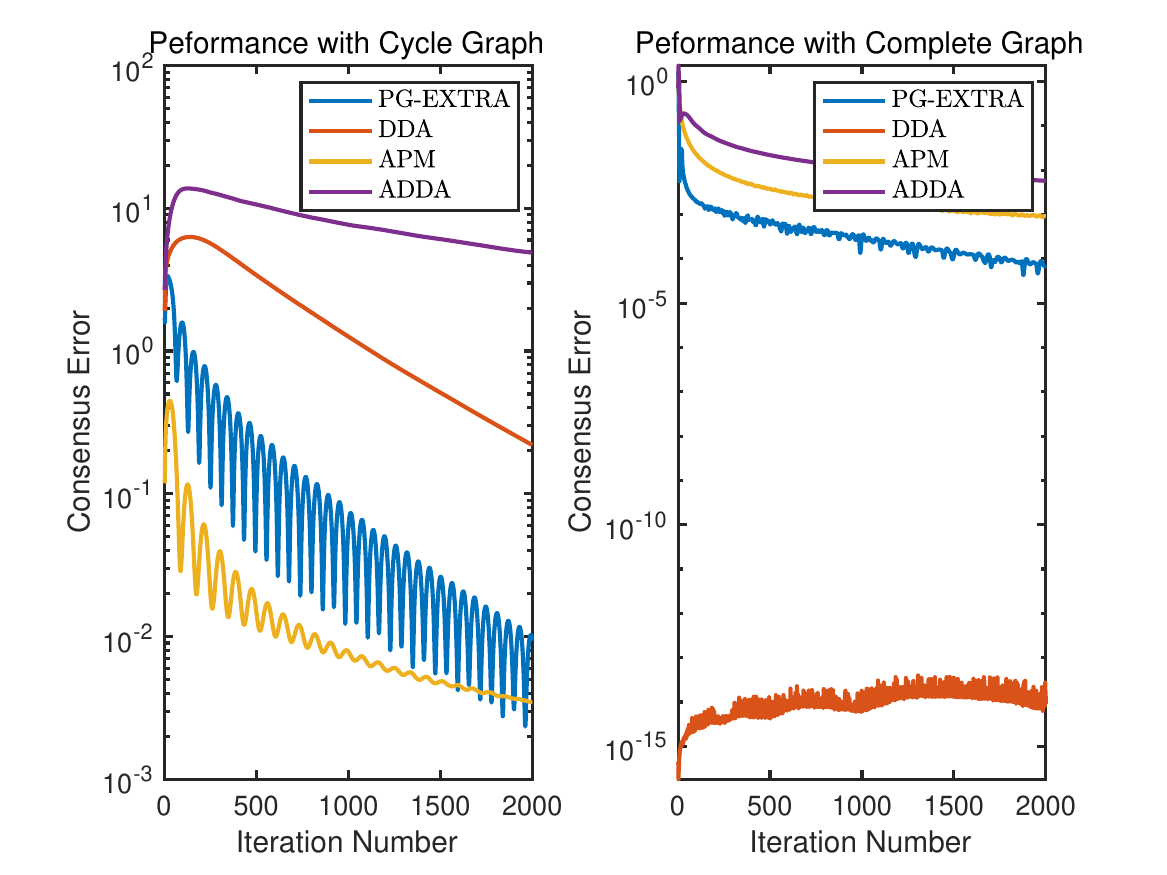}    
		\caption {Comparison of consensus error in Case I. }
		
		\label{fig:mse}
	\end{center}
\end{figure}


\begin{figure}
	\begin{center}
		\includegraphics[width=8.8cm]{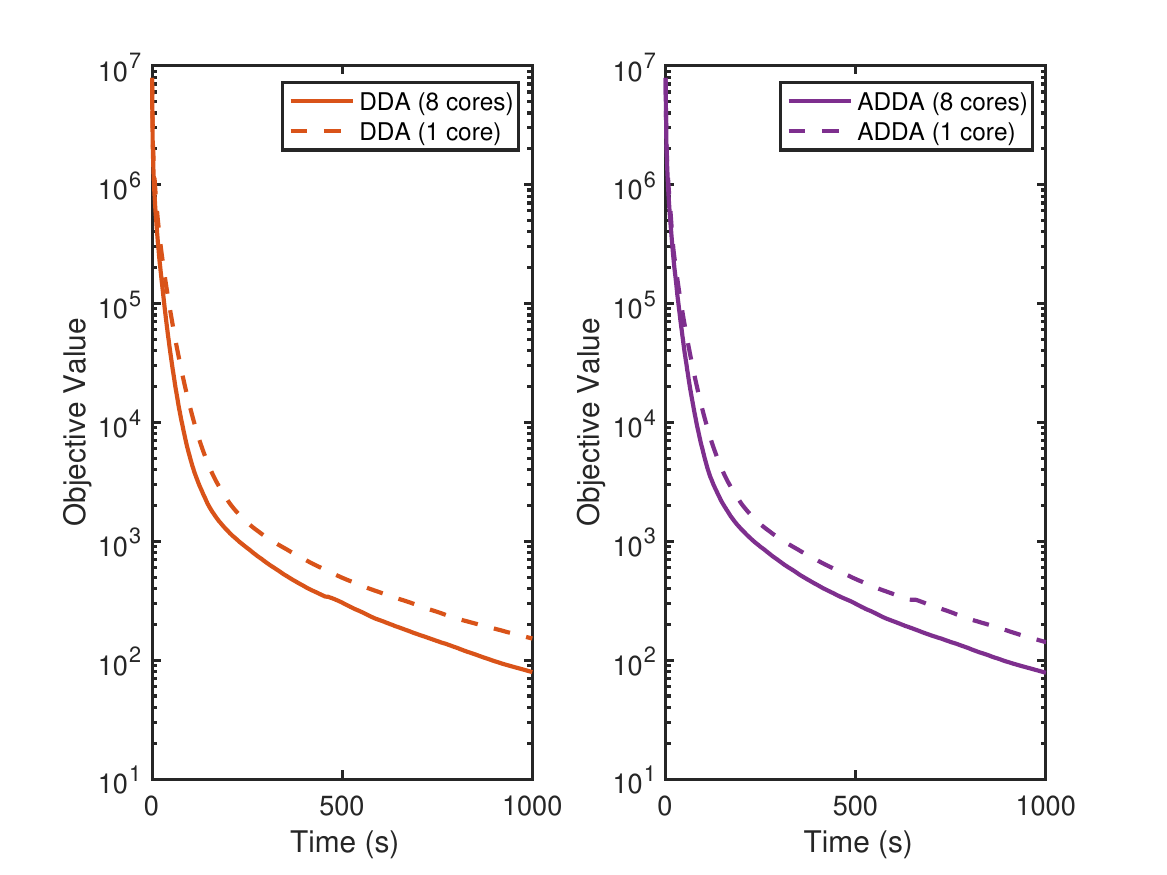}    
		\caption {Comparison of objective value in Case II. }
		
		\label{fig:robj}
	\end{center}
\end{figure}

\section{Conclusion}
In this work, we have designed two DDA algorithms for solving decentralized constrained optimization problems with improved rates of convergence. In the first one, a novel second-order dynamic average consensus scheme is developed, with which each agent locally generates a more accurate estimate of the dual variable than existing methods under mild assumptions. This property enables each agent to use large constant weight in the local dual average step, and therefore improves the rate of convergence. In the second algorithm, each agent retains the conventional first-order dynamic average consensus method to estimate the average of local gradients. Alternatively, the extrapolation technique together with the average consensus protocol is used to achieve acceleration over a decentralized network. 

This work opens several avenues for future research.
In this work, we focus on the basic setting with time-invariant bidirectional communication networks. We believe that the consensus-based dual averaging framework can be extended to tackle decentralized \emph{constrained} optimization in complex networks, e.g., directed networks \cite{pu2018push} and time-varying networks \cite{nedic2017achieving}.
Furthermore, we expect that our approach, as demonstrated by its centralized counterpart, i.e., follow-the-regularized-leader, may deliver superb performance in the online optimization setting \cite{yi2020distributed}. 
 


%

\appendices
\section{Roadmap for the proofs}
\begin{figure}
	\begin{center}
		\includegraphics[width=8.7cm]{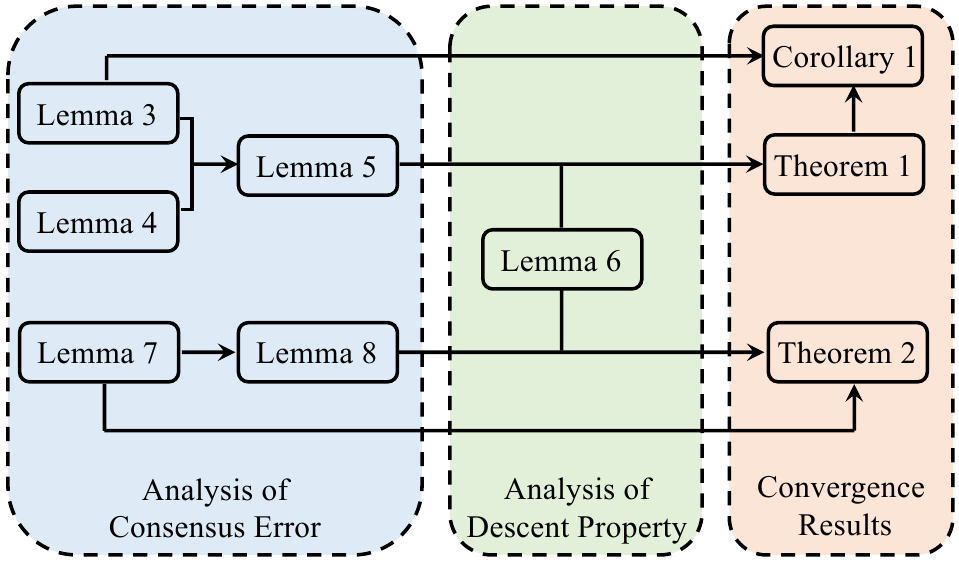}    
		\caption {Relation among the convergence results. }
		
		\label{roadmap}
	\end{center}
\end{figure}

Before proceeding to the proofs, we present Fig. \ref{roadmap} to illustrate how they relate to each other.

\section{Proof of Theorem \ref{inexact_thm}}

\subsection{Preliminaries}

We introduce several notations to facilitate the presentation of the proof. Let 
\begin{equation*}
	{\bf x}^{(t)}=\begin{bmatrix}
		x_{1}^{(t)} \\ \vdots \\ x_{n}^{(t)}
	\end{bmatrix}, \quad
	{\bf z}^{(t)}=\begin{bmatrix}
		z_{1}^{(t)} \\ \vdots \\ z_{n}^{(t)}
	\end{bmatrix},\quad {\bf s}^{(t)}=\begin{bmatrix}
		s_{1}^{(t)} \\ \vdots \\ s_{n}^{(t)}
	\end{bmatrix},
\end{equation*}

\begin{equation*}
	{\bf y}^{(t)}=\begin{bmatrix}
		y^{(t)} \\ \vdots \\ y^{(t)}
	\end{bmatrix}, \quad {\bf \nabla}^{(t)}=\begin{bmatrix}
		\nabla f_1(x_{1}^{(t)})\\ \vdots \\ 	\nabla f_n(x_{n}^{(t)})
	\end{bmatrix},
\end{equation*}

\begin{equation*}
	\overline{g}^{(t)}=\frac{1}{n}\sum_{i=1}^{n}\nabla f_i(x_{i}^{(t)}), \, \overline{s}^{(t)}=\frac{1}{n}\sum_{i=1}^{n}s_i^{(t)}, \, \overline{z}^{(t)}=\frac{1}{n}\sum_{i=1}^{n}z_i^{(t)}.
\end{equation*}
Using these notations, we express \eqref{2nd_order_consensus} in the following compact form
\begin{subequations}
	\begin{align}
		{\bf z}^{(t)} &= {\bf P} \left(  {\bf z}^{(t-1)} +   {\bf s}^{(t-1)}  \right), \label{consensus_compact_z}\\	
		{\bf s}^{(t)} &= {\bf P}   {\bf s}^{(t-1)} +  \nabla^{(t)}-\nabla^{(t-1)}, \label{consensus_compact_s}
	\end{align}
\end{subequations}
where ${\bf P} = P\otimes I$. Before proceeding to the proof of Theorem \ref{inexact_thm}, we present several technical lemmas. 

Recall a lemma from \cite{xu2015augmented}.
\begin{lemma}
	\label{pertubed_consensus}
	Suppose that $\{\varepsilon^{(t)}\}_{t\geq 0}$ and $\{\epsilon^{(t)}\}_{t\geq 0}$ are two sequences of positive scalars such that for all $t\geq 0$, 
	\begin{equation*}
		\varepsilon^{(t)}\leq \delta^{t} c+\sum_{\tau=0}^{t-1}\delta^{t-\tau-1}\epsilon^{(\tau)}
	\end{equation*}
	where $\delta\in(0,1)$ and $c\geq 0$ is a constant. Then, the following holds for all $t\geq 1$:
	\begin{equation*}
		\sum_{\tau=1}^{t}(\varepsilon^{(\tau)})^2\leq  \frac{2}{(1-\delta)^2}\sum_{\tau=0}^{t-1}(\epsilon^{(\tau)})^2+\frac{2c^2}{1-\delta^2}.
	\end{equation*}
\end{lemma}


\begin{lemma} \label{conservation_property}
	For Algorithm \ref{FDDA}, we have that for any $t\geq 0$
	\begin{equation}\label{conservation}
		\overline{s}^{(t)}=\overline{g}^{(t)}, \quad 	\overline{z}^{(t+1)}=\sum_{\tau=0}^{t}\overline{s}^{(\tau)}.
	\end{equation}
\end{lemma} 
\begin{proof}[Proof of Lemma \ref{conservation_property}]
	We prove by induction.
	For $t=0$, we readily have \eqref{conservation} satisfied since $s_i^{(0)}=\nabla f_i(x^{(0)})$ and $z_i^{0}=0$ for all $i$. Suppose that \eqref{conservation} holds for $t-1$. Using \eqref{consensus_compact_s},
	\begin{equation}\label{kron}
		(A\otimes B)(C\otimes D) = (AC\otimes BD),
	\end{equation}
	and the double stochasticity of $P$,
	we have
	\begin{equation*}
		\begin{split}
			\overline{s}^{(t)} & =  \frac{1}{n}({\bf 1}^{\mathrm{T}}\otimes I) \left(   {\bf P}{\bf s}^{(t-1)} +\nabla^{(t)}-\nabla^{(t-1)} \right) \\
			& = \frac{1}{n}({\bf 1}^{\mathrm{T}}\otimes I) \left(   ({P}\otimes I){\bf s}^{(t-1)} +\nabla^{(t)}-\nabla^{(t-1)} \right) \\
			& = \frac{1}{n}\left(({\bf 1}^{\mathrm{T}} P)\otimes I\right) {\bf s}^{(t-1)} +\overline{g}^{(t)}-\overline{g}^{(t-1)}   \\
			&= \overline{s}^{(t-1)}+ \overline{g}^{(t)}-\overline{g}^{(t-1)}  = \overline{g}^{(t)}.
		\end{split}
	\end{equation*}
	Upon using a similar argument for $\overline{z}^{(t)}$, we have
	\begin{equation*}
		\begin{split}
			\overline{z}^{(t+1)} & =  \frac{1}{n}({\bf 1}^{\mathrm{T}}\otimes I)    {(P\otimes I)}\left({\bf z}^{(t)} +{\bf s}^{(t)}  \right) \\
			&= \overline{z}^{(t)}+  \overline{s}^{(t)}  = \sum_{\tau =0}^{t}\overline{s}^{(\tau)}.
		\end{split}
	\end{equation*}
	Therefore, \eqref{conservation} holds for all $t$.
\end{proof}

\begin{lemma} \label{spectral_radius_M}
	If the parameter $a$ satisfies \eqref{1st_condition},
	then $\rho({\bf M})<1$, where ${\bf M}$ is defined in \eqref{def:M}.
	
\end{lemma} 
\begin{proof}[Proof of Lemma \ref{spectral_radius_M}]
	{	Recall the two real eigenvalues of ${\bf M}$ in \eqref{eig}.
		Since $\xi_1>0$ and $\xi_2>0$, we have $\rho({\bf M}) = \lvert\lambda_1  \rvert >\lvert \lambda_2 \rvert$. Also, one can verify that $\rho({\bf M})$ monotonically increases with $a$.
		The solution to the equation $\lvert \lambda_1 \rvert =1$ can be identified as $ a  =(1-\beta)^2/(2\beta L)$. Therefore, we conclude that $\rho({\bf M})<1$ as long as \eqref{1st_condition} is satisfied.	 
		%
	}
\end{proof}

The following lemma establishes the relation between sequences $\{x_{i}^{(t)}\}_{t\geq 0}$ and $\{{y}^{(t)}\}_{t\geq 0}$.


\begin{lemma}\label{consensus_error_lemma}
	Suppose that Assumptions \ref{lipschitzassumption}, \ref{graphconnected} hold and the parameter $a$ in Algorithm \ref{FDDA} satisfies \eqref{1st_condition}. Then, for all $t\geq 0$, it holds that
	\begin{equation}\label{consensus_error}
		\begin{split}
			&	\sum_{\tau=1}^{t}\lVert {\bf x}^{(\tau)}-\mathbf{y}^{(\tau)}\rVert^2 \\
			&\leq \frac{8}{9\left(1-\rho({\bf M})\right)^2}\sum_{\tau=0}^{t-1}\lVert  {\bf y}^{(\tau+1)}-{\bf y}^{(\tau)} \rVert^2 + \frac{8\pi^2 }{9L^2\left(1-(\rho({\bf M}))^2\right)}
		\end{split}
	\end{equation}
	where 
	\begin{equation*}
		\pi^2= {\sum_{i=1}^n\left\|\nabla f_i(x^{(0)}) - \overline{g}^{(0)}\right\|^2} .
	\end{equation*}
\end{lemma}   

\begin{proof}[Proof of Lemma \ref{consensus_error_lemma}]
	From Lemma \ref{conservation_property}, we have
	\begin{equation*}
		\overline{z}^{(\tau)} = \overline{z}^{(\tau-1)}+\overline{g}^{(\tau-1)}.
	\end{equation*}
	Define
	\begin{equation}\label{error_notation}
		\tilde{ {\bf s}}^{(t)}= {\bf s}^{(t)}-{\bf 1}\otimes \overline{s}^{(t)}, \, 	\tilde{ {\bf z}}^{(t)}= {\bf z}^{(t)}-{\bf 1}\otimes \overline{z}^{(t)}.
	\end{equation}
	These in conjunction with \eqref{consensus_compact_z} lead to
	\begin{equation}\label{deviation_z}
		\begin{split}
			\tilde{\bf z}^{(\tau)} 
			=  \mathbf{P}{\bf z}^{(\tau-1)}-\mathbf{1}\otimes\overline{z}^{(\tau-1)} +\mathbf{P}{\bf s}^{(\tau-1)} - \mathbf{1}\otimes\overline{s}^{(\tau-1)}.
		\end{split}
	\end{equation}
	Because of $\mathbf{1}\otimes\bar{z}^{(\tau-1)} = (\mathbf{1}\otimes I)\bar{z}^{(\tau-1)}$ and \eqref{kron}, 
we have
	$$ \mathbf{1}\otimes\bar{z}^{(\tau-1)} = \frac{1}{n}(\mathbf{1}\otimes I)(\mathbf{1}^{\mathrm{T}}\otimes I)\mathbf{z}^{(\tau-1)} = \left(\frac{\mathbf{1}\mathbf{1}^{\mathrm{T}}}{n}\otimes I\right)\mathbf{z}^{(\tau-1)}. $$
	In addition, we have
	\begin{align}\label{consensus_transformation}
		&\mathbf{P}{\bf z}^{(\tau-1)}-\mathbf{1}\otimes\overline{z}^{(\tau-1)} = (P\otimes I)\mathbf{z}^{(\tau-1)} - \left(\frac{\mathbf{1}\mathbf{1}^{\mathrm{T}}}{n}\otimes I\right)\mathbf{z}^{(\tau-1)} \nonumber \\
		& = \left(\left(P-\frac{\mathbf{1}\mathbf{1}^{T}}{n}\right)\otimes I\right){\bf z}^{(\tau-1)} \nonumber\\
		& = \left(\left(P-\frac{\mathbf{1}\mathbf{1}^{T}}{n}\right)\otimes I\right)\left(\tilde{\bf z}^{(\tau-1)} + (\mathbf{1}\otimes I)\bar{z}^{(\tau-1)}\right) \nonumber\\
		& = \left(\left(P-\frac{\mathbf{1}\mathbf{1}^{T}}{n}\right)\otimes I\right)\tilde{\bf z}^{(\tau-1)} + \left(\left(P\mathbf{1}-\mathbf{1}\right)\otimes I\right)\bar{z}^{(\tau-1)} \nonumber\\
		& = \left(\left(P-\frac{\mathbf{1}\mathbf{1}^{T}}{n}\right)\otimes I\right)\tilde{\bf z}^{(\tau-1)},
	\end{align}
	where the last equality is due to the double stochasticity of $P$. 
	Using the same arguments as above for $\mathbf{P}{\bf s}^{(\tau-1)}-\mathbf{1}\otimes\overline{s}^{(\tau-1)}$
	and Assumption \ref{graphconnected}, we have
	\begin{equation}\label{eq:take-exp}
		\begin{split}
			\lVert	\tilde{\bf z}^{(\tau)} \rVert 
			&\leq  \left\lVert \mathbf{P}{\bf z}^{(\tau-1)} - \mathbf{1}\otimes\overline{z}^{(\tau-1)}\right\rVert+\left\lVert \mathbf{P}{\bf s}^{(\tau-1)}-\mathbf{1}\otimes\overline{s}^{(\tau-1)}\ \right\rVert \\
			& \leq \beta 	\lVert	\tilde{\bf z}^{(\tau-1)} \rVert +\beta 	\lVert	\tilde{\bf s}^{(\tau-1)} \rVert .
		\end{split}
	\end{equation}
	Similarly, from Lemma \ref{conservation_property}, we obtain
	\begin{equation}\label{eq:tilde-s}
		\begin{split}
			\lVert	\tilde{\bf s}^{(\tau)} \rVert
			& =\left\| \mathbf{P}{{\bf s}}^{(\tau-1)}-\left(\mathbf{1}\otimes I\right)\overline{s}^{(\tau-1)}+\nabla^{(\tau)}- \nabla^{(\tau-1)}\right\rVert \\
			& \leq  \beta 	\lVert	\tilde{\bf s}^{(\tau-1)} \rVert+\left\| \nabla^{(\tau)}- \nabla^{(\tau-1)}\right\| \\
			& \leq  \beta 	\lVert	\tilde{\bf s}^{(\tau-1)} \rVert+L\lVert  {\bf x}^{(\tau)}-{{\bf x}}^{(\tau-1)}\rVert
		\end{split}
	\end{equation}
	where the last inequality is due to Assumption \ref{lipschitzassumption}. 
	Using 
	\begin{equation*}
		\begin{split}
			\lVert{\bf x}^{(\tau)}-{{\bf x}}^{(\tau-1)} \rVert
			\leq& \lVert {\bf x}^{(\tau)}-{{\bf y}}^{(\tau)}\rVert+\lVert {\bf x}^{(\tau-1)}-{{\bf y}}^{(\tau-1)}\rVert \\
			&+ \lVert  {\bf y}^{(\tau)}-{{\bf y}}^{(\tau-1)}\rVert,
		\end{split}
	\end{equation*}
	and Lemma \ref{gamma_continuity},
	we obtain
	\begin{equation*}
		\begin{split}
			&\lVert{\bf x}^{(\tau)}-{{\bf x}}^{(\tau-1)} \rVert
			\leq  a\lVert\tilde{\mathbf z}^{(\tau)}\rVert+ a\lVert\tilde{\mathbf z}^{(\tau-1)}\rVert+ \lVert  {\bf y}^{(\tau)}-{{\bf y}}^{(\tau-1)}\rVert \\
			& \leq a(\beta+1) \lVert\tilde{\mathbf z}^{(\tau-1)}\rVert+a\beta \lVert\tilde{\mathbf s}^{(\tau-1)}\rVert+ \lVert  {\bf y}^{(\tau)}-{{\bf y}}^{(\tau-1)}\rVert.
		\end{split}
	\end{equation*}
	Upon substituting the above inequality into \eqref{eq:tilde-s}, we obtain
	\begin{equation}\label{gradient_error}
		\begin{split}
			\lVert	\tilde{\bf s}^{(\tau)} \rVert
			\leq &
			\beta(aL+1)	\lVert	\tilde{\bf s}^{(\tau-1)} \rVert 
			+aL(\beta+1) 	\lVert	\tilde{\bf z}^{(\tau-1)} \rVert\\
			&	+L\lVert  {\bf y}^{(\tau)}-{{\bf y}}^{(\tau-1)}\rVert.
		\end{split}
	\end{equation}
	By combining \eqref{eq:take-exp} and \eqref{gradient_error}, we establish the following inequality:
	\begin{equation*}
		\begin{split}
			\begin{bmatrix}
				\lVert	\tilde{\bf z}^{(\tau)} \rVert\\
				\lVert	\tilde{\bf s}^{(\tau)} \rVert
			\end{bmatrix} 
			\leq  \mathbf{M} 	\begin{bmatrix}
				\lVert	\tilde{\bf z}^{(\tau-1)} \rVert \\
				\lVert	\tilde{\bf s}^{(\tau-1)} \rVert
			\end{bmatrix}   + L \begin{bmatrix}
				0\\
				\lVert  {\bf y}^{(\tau)}-{{\bf y}}^{(\tau-1)}\rVert
			\end{bmatrix} 
		\end{split}
	\end{equation*}
	where $\mathbf{M}$ is defined in \eqref{def:M}.
	By iterating the above inequality and using
	\begin{equation*}
		\lVert	\tilde{\bf z}^{(0)} 	\lVert= 0 ,\quad 	\lVert	\tilde{\bf s}^{(0)}	\lVert = \sqrt{\sum_{i=1}^n\left\|\nabla f_i(x^{(0)}) - \overline{g}^{(0)}\right\|^2}:=\pi ,
	\end{equation*}
	we obtain
	\begin{equation*}
		\begin{split}
			\begin{bmatrix}
				\lVert	\tilde{\bf z}^{(t)} \rVert \\
				\lVert	\tilde{\bf s}^{(t)} \rVert
			\end{bmatrix}
			\leq&L \sum_{\tau=0}^{t-1}\mathbf{M}^{t-\tau-1} \begin{bmatrix}
				0\\
				\lVert  {\bf y}^{(\tau+1)}-{{\bf y}}^{(\tau)}\rVert
			\end{bmatrix} +   \mathbf{M} ^{t} \begin{bmatrix}
				0 \\
				\pi
			\end{bmatrix} .
		\end{split}
	\end{equation*}
	Recall the eigenvalues of matrix $\mathbf{M}$ in \eqref{eig}.
	Then an analytical form can be presented for the $n$th power of $\mathbf{M}$ (see, e.g., \cite{williams1992n})
	\begin{equation*}
		\begin{split}
			\mathbf{M}^n = \lambda_1^n\left( \frac{\mathbf{M}- \lambda_2 I }{\lambda_1-\lambda_2} \right) -\lambda_2^n\left( \frac{\mathbf{M}- \lambda_1 I }{\lambda_1-\lambda_2} \right) .
		\end{split}
	\end{equation*}
Therefore, the $(1,2)$-th entry of $\mathbf{M}^n$ can be written as
	\begin{equation*}
		(\mathbf{M}^n)_{12}=\frac{ \mathbf{M}_{12}(\lambda_1^n  -\lambda_2^n)}{\lambda_1-\lambda_2} = \frac{ \beta(\lambda_1^n  -\lambda_2^n)}{\lambda_1-\lambda_2} \leq \frac{2\beta(\rho(\mathbf{M}))^n}{\xi_2}.
	\end{equation*} 
	Due to the assumption that
	$
	{1}/{a}>{2\beta L}/{(1-\beta)^2}
	$
	and $\beta\in (0,1)$,
	we have 
	\begin{equation*}
		\xi_2=\beta a L\sqrt{1+\frac{4(\beta+1)}{\beta a L}} >\beta a L\sqrt{1+\frac{8(\beta+1)}{(1-\beta)^2}} >3\beta a L.
	\end{equation*}
	Therefore,
	\begin{equation}\label{consensus_error_mid}
		\begin{split}
			&	\lVert	\tilde{\bf z}^{(t)} \rVert\\
			& \leq
			\frac{2\beta{L}}{\xi_2}\sum_{\tau=0}^{t-1}\left( \rho(\mathbf{M}) \right)^{t-\tau-1}	\lVert {\bf y}^{(\tau+1)}-{\bf y}^{(\tau)}\rVert+ 	\frac{2\beta}{\xi_2} \left( \rho(\mathbf{M}) \right)^{t} \pi\\
			&	\leq \frac{2}{3a}\sum_{\tau=0}^{t-1}\left( \rho(\mathbf{M}) \right)^{t-\tau-1}	\lVert {\bf y}^{(\tau+1)}-{\bf y}^{(\tau)}\rVert+ \frac{2}{3aL}\left( \rho(\mathbf{M}) \right)^{t}\pi.
		\end{split}
	\end{equation}
	Using Lemma \ref{gamma_continuity}, we further obtain
	\begin{equation}\label{x-y}
		\begin{split}
			&\lVert{{\bf x}}^{(t)}-{{\bf y}}^{(t)}\rVert \leq   	  a 	\lVert	\tilde{\bf z}^{(t)} \rVert \\
			&	\leq \frac{2}{3}\sum_{\tau=0}^{t-1}\left( \rho(\mathbf{M}) \right)^{t-\tau-1}	\lVert {\bf y}^{(\tau+1)}-{\bf y}^{(\tau)}\rVert+ \frac{2}{3L}\left( \rho(\mathbf{M}) \right)^{t}\pi.
		\end{split}
	\end{equation} 
Further using the above relation and Lemma \ref{pertubed_consensus}, 
	the inequality \eqref{consensus_error} follows as desired.
\end{proof}

Then, a lemma is stated for the prox-mapping. 

\begin{lemma}\label{inexact_dual_averaging_inter}
	Given a sequence of variables $\{ \zeta^{(t)}\}_{t\geq0}$ and a positive sequence $\{ a_t\}_{t\geq0}$, 
	for $\{{\nu}^{(t)}\}_{t\geq 0}$ generated by
	\begin{equation*}
		{\nu}^{(t)}= \nabla d^*\left(-\sum_{\tau=1}^{t}a_{\tau} \zeta^{(\tau)}\right),
	\end{equation*}
	where $\nu^{(0)} = x^{(0)}$ in \eqref{eq:initial-cond}, it holds that
	\begin{equation}\label{inexact}
		\begin{split}
			\sum_{\tau=1}^{t}a_\tau\left\langle \zeta^{(\tau)},{\nu}^{(\tau)}-x^*\right \rangle \leq d(x^*)-\sum_{\tau=1}^{t}\frac{1}{2}\lVert  {\nu}^{(\tau)}-{\nu}^{(\tau-1)} \rVert^2.
		\end{split}
	\end{equation}
\end{lemma}   
\begin{proof}[Proof of Lemma \ref{inexact_dual_averaging_inter}]
	Define 
	\begin{equation*}
		\begin{split}
			m_{t}({x})= \sum_{\tau=1}^{t}a_{\tau}\left\langle \zeta^{(\tau)},x\right \rangle +d({x})
		\end{split}       
	\end{equation*}
	where $m_0(x)= d(x)$.
	Since 
	\begin{equation*}
		\nu^{(\tau-1)} = \argmin_{x\in\mathcal{X}} m_{\tau-1}(x)
	\end{equation*}
	and
	$m_{\tau-1}(x)$ is strongly convex with modulus $1$, we have
	\begin{equation*}
		\begin{split}
			m_{\tau-1}(x)-m_{\tau-1}({\nu}^{(\tau-1)})\geq  \frac{1}{2}\lVert x- \nu^{(\tau-1)} \rVert^2, \forall x\in\mathcal{X}.
		\end{split}
	\end{equation*}
	Upon taking $x= \nu^{(\tau)}$ in the above inequality and using 
	\begin{equation*}
		m_\tau({x}) = m_{\tau-1}({x})+a_{\tau}\left\langle \zeta^{(\tau)},x\right \rangle,
	\end{equation*}
	we obtain
	\begin{equation*}
		\begin{split}
			0\leq& m_{\tau-1}({\nu}^{(\tau)})-m_{\tau-1}({\nu}^{(\tau-1)})-\frac{1}{2}\lVert  {\nu}^{(\tau)}-{\nu}^{(\tau-1)} \rVert^2\\
			=& m_\tau({\nu}^{(\tau)})-a_{\tau} \left  \langle \zeta^{(\tau)},{\nu}^{(\tau)} \right\rangle -m_{\tau-1}({\nu}^{(\tau-1)})\\
			&-\frac{1}{2}\lVert \nu^{(\tau)}-{\nu}^{(\tau-1)} \rVert^2,
		\end{split}
	\end{equation*}
	which is equivalent to
	\begin{equation*}
		\begin{split}
			a_{\tau}\left\langle  \zeta^{(\tau)},{\nu}^{(\tau)} \right\rangle\leq&  m_\tau({\nu}^{(\tau)}) -m_{\tau-1}({\nu}^{(\tau-1)})\\
			&-\frac{1}{2}\lVert {\nu}^{(\tau)}-{\nu}^{(\tau-1)} \rVert^2 .
		\end{split}
	\end{equation*}
	Iterating the above equation from $\tau=1$ to $\tau =t$ yields
	\begin{equation}\label{inexact_gradient_result1}
		\begin{split}
			&\sum_{\tau=1}^{t}\left\langle a_{\tau} \zeta^{(\tau)},{\nu}^{(\tau)} \right\rangle\\
			\leq&  m_{t}({\nu}^{(t)}) -m_{0}({\nu}^{(0)})-\sum_{\tau=1}^{t}\frac{1}{2}\lVert  {\nu}^{(\tau)}-{\nu}^{(\tau-1)} \rVert^2 \\
			=& m_{t}({\nu}^{(t)}) -\sum_{\tau=1}^{t}\frac{1}{2}\lVert  {\nu}^{(\tau)}-\nu^{(\tau-1)} \rVert^2
		\end{split}
	\end{equation}
	We turn to consider
	\begin{equation*}
		\begin{split}
			&\sum_{\tau=1}^{t}a_{\tau} \left\langle \zeta^{(\tau)},-x^*\right \rangle\\
			&\leq  \max_{{x}\in\mathcal{X}} \left\{  \sum_{\tau=1}^{t} a_{\tau}\left\langle \zeta^{(\tau)},-x \right\rangle  -d(x)\right\}+d(x^*) \\
			& = -\min_{{x}\in\mathcal{X}} \left\{  \sum_{\tau=1}^{t} a_{\tau} \left\langle  \zeta^{(\tau)},x \right\rangle+d(x)\right\}+d(x^*)\\
			& = -m_{t}({\nu}^{(t)})+d(x^*),
		\end{split}
	\end{equation*}
	which together with \eqref{inexact_gradient_result1} leads to the inequality in \eqref{inexact}, thereby concluding the proof.
\end{proof}

\subsection{Proof of Theorem \ref{inexact_thm}}

\begin{proof}[Proof of Theorem \ref{inexact_thm}]   
	For all $\tau \geq 1$, we have
	\begin{align}
		& \frac{1}{n}\sum_{i=1}^{n}a\left( f_i({y}^{(\tau)}) -f_i(x^*)\right) \nonumber \\
		& \leq \frac{1}{n}\sum_{i=1}^{n} a\Big( f_i(x_{i}^{(\tau-1)}) - f_i(x^*) + \frac{L}{2}\lVert {y}^{(\tau)}-x_{i}^{(\tau-1)} \rVert^2 \nonumber \\
		& \qquad \qquad \quad +\left\langle \nabla f_i(x_{i}^{(\tau-1)}),{y}^{(\tau)}-x_{i}^{(\tau-1)} \right\rangle\Big) \nonumber \\
		& \leq \frac{1}{n}\sum_{i=1}^{n}  a\Big(\frac{L}{2}\lVert {y}^{(\tau)}-x_{i}^{(\tau-1)} \rVert^2 + \left\langle \nabla f_i(x_{i}^{(\tau-1)}),{y}^{(\tau)}-x^*\right \rangle \Big) \nonumber \\
		& = \frac{1}{n}\sum_{i=1}^{n}a\Big(\frac{L}{2}\lVert {y}^{(\tau)}-x_{i}^{(\tau-1)} \rVert^2 \Big) + a\left\langle \overline{g}^{(\tau-1)}, {y}^{(\tau)}-x^*\right\rangle,
		\label{FDDA_smoothness_use}
	\end{align}
	where the two inequalities are due to Assumption \ref{lipschitzassumption}.
	By iterating \eqref{FDDA_smoothness_use} from $\tau = 1$ to $\tau = t$ and using Lemma \ref{inexact_dual_averaging_inter} ($a_\tau =a, \zeta^{(\tau)}=\overline{g}^{(\tau-1)}$, and $\nu^{(\tau)}=y^{(\tau)}$), we obtain
	\begin{align}
		& \sum_{\tau=1}^{t} a \left( f(y^{(\tau)}) - f(x^*) \right) \nonumber \\
		& \leq \frac{1}{n}\sum_{\tau=1}^{t}\sum_{i=1}^{n}a\Big(\frac{L}{2}\lVert {y}^{(\tau)}-x_{i}^{(\tau-1)} \rVert^2 \Big) \nonumber \\ 
		& \quad -\frac{1}{2}\sum_{\tau=1}^{t}\lVert y^{(\tau)} - y^{(\tau-1)}\rVert^2 + d(x^*).  \nonumber \\
		& = \frac{aL}{2n}\sum_{\tau=1}^{t}\lVert {\bf y}^{(\tau)}-{\bf x}^{(\tau-1)} \rVert^2  -\frac{1}{2}\sum_{\tau=1}^{t}\lVert y^{(\tau)} - y^{(\tau-1)}\rVert^2 + d(x^*).  \nonumber
	\end{align}
	Then we use the inequality
	$$ \|\mathbf{y}^{(\tau)} - \mathbf{x}^{(\tau-1)}\|^2 \leq 2\|\mathbf{y}^{(\tau)} - \mathbf{y}^{(\tau-1)}\|^2 + 2\|\mathbf{y}^{(\tau-1)} - \mathbf{x}^{(\tau-1)}\|^2 $$
	and the convexity of $f$ to get
	\begin{align*}
		& at\left(f(\tilde{y}^{(t)})  -f(x^*)\right) \\
		& \leq \frac{1}{n}\sum_{\tau=1}^{t} \left(aL- \frac{1 }{2}\right)\|\mathbf{y}^{(\tau)} - \mathbf{y}^{(\tau-1)}\|^2\\  
		& \quad + \frac{aL}{n}\sum_{\tau=1}^{t} \|\mathbf{x}^{(\tau-1)} - \mathbf{y}^{(\tau-1)}\|^2 + d(x^*).
	\end{align*}
	Upon using Lemma \ref{consensus_error_lemma}, one has
	\begin{align}
		& a t\left( f(\tilde{y}^{(t)})- f(x^*)\right) + \frac{\gamma}{n}\sum_{\tau=1}^{t}\|\mathbf{y}^{(\tau)} - \mathbf{y}^{(\tau-1)}\|^2 \nonumber \\
		& \leq d(x^*) +\frac{8a\pi^2 }{9nL\left(1-(\rho({\bf M}))^2\right)}= C, \label{eq:last}
	\end{align}
	where $\gamma>0$ is defined in \eqref{def:gamma}. Therefore we have \eqref{FDDA_rate} as desired.

	From \eqref{eq:last} and $f(\tilde{y}^{(t)})  - f(x^*)\geq 0$, we have 
	$$ \sum_{\tau=1}^{t}\|\mathbf{y}^{(\tau)} - \mathbf{y}^{(\tau-1)}\|^2\leq \frac{nC}{\gamma}. $$
By using the above inequality, the convexity of $\lVert \cdot \rVert^2$, Jensen's Inequality, and Lemma \ref{consensus_error_lemma}, we arrive at
	\begin{align*}
		&	t\lVert{\tilde{\bf x}}^{(t)}-\tilde{\bf y}^{(t)}\rVert^2 \leq \sum_{\tau=1}^{t}\lVert{{\bf x}}^{(\tau)}-{{\bf y}}^{(\tau)}\rVert^2 \\
		& \leq \frac{8}{9\left(1-\rho({\bf M})\right)^2}\sum_{\tau=0}^{t-1}\lVert  {\bf y}^{(\tau+1)}-{\bf y}^{(\tau)} \rVert^2 + \frac{8\pi^2 }{9L^2\left(1-(\rho({\bf M}))^2\right)}\\
		& \leq \frac{8nC}{9\gamma\left(1-\rho({\bf M})\right)^2} + \frac{8\pi^2 }{9L^2\left(1-(\rho({\bf M}))^2\right)} = D,
	\end{align*}
	which implies \eqref{eq:x-y-dist} as desired.
\end{proof}

\section{Proof of Corollary \ref{unconstrained corollary}}
\begin{proof}[Proof of Corollary \ref{unconstrained corollary}]
	We consider
	\begin{equation}\label{unconstrained}
		\begin{split}
			& f(x_{i}^{(\tau)})-f(y^{(\tau)})\\
			& \leq \frac{1}{n}\sum_{j=1}^{n} \left( f_j(x_{i}^{(\tau)})-\langle \nabla f_j(x_{j}^{(\tau)}),y^{(\tau)}-x_{j}^{(\tau)} \rangle- f_j(x_{j}^{(\tau)}) \right)\\
			& \leq  \frac{1}{n}\sum_{j=1}^{n} \Big( \langle \nabla f_j(x_{j}^{(\tau)}),x_{i}^{(\tau)}-x_{j}^{(\tau)} \rangle-\langle \nabla f_j(x_{j}^{(\tau)}),y^{(\tau)}-x_{j}^{(\tau)} \rangle\\
			&\quad\quad\quad \quad  +\frac{L}{2}\lVert x_{i}^{(\tau)}-y^{(\tau)}+y^{(\tau)}-x_{j}^{(\tau)}  \rVert^2 \Big) \\
			& =  \frac{1}{n}\sum_{j=1}^{n} \Big( \langle \nabla f_j(x_{j}^{(\tau)}),x_{i}^{(\tau)}-y^{(\tau)} \rangle+{L}\lVert x_{i}^{(\tau)}-y^{(\tau)}  \rVert^2 \\
			& \quad \quad \quad \quad +{L}\lVert y^{(\tau)}-x_{j}^{(\tau)}  \rVert^2 \Big) \\
			& =  \left\langle \overline{g}^{(\tau)},x_{i}^{(\tau)}-y^{(\tau)} \right\rangle+{L}\lVert x_{i}^{(\tau)}-y^{(\tau)}  \rVert^2 +\frac{L}{n}\lVert \mathbf{y}^{(\tau)}-\mathbf{x}^{(\tau)}  \rVert^2,
		\end{split}
	\end{equation}
	where the two inequalities follow from Assumption \ref{lipschitzassumption}.
	When $\mathcal{X}=\mathbb{ R}^m$,
	the closed-form solutions for \eqref{primal_smooth} and \eqref{auxiliary_sequence} can be identified as
	\begin{equation*}
		x_{i}^{(\tau)}=-a{z_{i}^{(\tau)}}, \quad
		y^{(\tau)}=-a{\overline{z}^{(\tau)}},
	\end{equation*}
	implying that $y^{(\tau)}=\frac{1}{n}\sum_{i=1}^{n}x_i^{(t)}$.
	Upon summing up \eqref{unconstrained} from $i= 1$ to $i = n$, we obtain
	\begin{equation}\label{convergence:xtoy}
		\begin{split}
			\sum_{i=1}^{n}\left(f(x_{i}^{(\tau)})-f(y^{(\tau)})\right) 
			\leq 2L\lVert \mathbf{y}^{(\tau)}-\mathbf{x}^{(\tau)}  \rVert^2.
		\end{split}
	\end{equation}
	Then, we iterate \eqref{convergence:xtoy} from $\tau = 1$ to $\tau = t$ and use the convexity of $f$ to get
	\begin{equation}
		\begin{split}
			&t\sum_{i=1}^{n}\left(f(\tilde{x}_{i}^{(t)})-f(y^{(\tau)})\right) \leq  \sum_{\tau=1}^{t} \sum_{i=1}^{n}\left(f(x_{i}^{(\tau)})-f(y^{(\tau)})\right) \\
			&	\leq 2L\sum_{\tau=1}^{t}\lVert \mathbf{y}^{(\tau)}-\mathbf{x}^{(\tau)}  \rVert^2.
		\end{split}
	\end{equation}
	where $\tilde{x}_{i}^{(t)}= {t}^{-1}\sum_{\tau=1}^{t}x_i^{(\tau)}$.
	Using Lemma \ref{consensus_error_lemma}, we obtain
	\begin{equation}\label{convegence_x}
		\begin{split}
			&	t\sum_{i=1}^{n}\left(f(\tilde{x}_{i}^{(t)})-f(y^{(\tau)})\right) \\
			&	\leq  \frac{16L}{9\left(1-\rho({\bf M})\right)^2}\sum_{\tau=0}^{t-1}\lVert  {\bf y}^{(\tau+1)}-{\bf y}^{(\tau)} \rVert^2 + \frac{16\pi^2 }{9L\left(1-(\rho({\bf M}))^2\right)}.
		\end{split}
	\end{equation}
	Recall \eqref{eq:last}, we have
	\begin{align}\label{recall_convergence_result}
		& a t\left( f(\tilde{y}^{(t)})- f(x^*)\right)  
		\leq  d(x^*) +\frac{8a\pi^2 }{9nL\left(1-(\rho({\bf M}))^2\right)} \nonumber \\
		&-\left(\frac{1}{2} -aL-\frac{8aL}{9\left(1-(\rho({\bf M}))^2\right)}\right)\sum_{\tau=1}^{t}\|{y}^{(\tau)} - {y}^{(\tau-1)}\|^2 .
	\end{align}
	By multiplying $n/a>0$ on both sides of \eqref{recall_convergence_result} and adding the resultant inequality to \eqref{convegence_x}, we get
	\begin{equation}
		\begin{split}
			&	t\left( f(\tilde{x}_{i}^{(t)})-f(x^*)\right)\leq 
			t\sum_{i=1}^{n}\left(f(\tilde{x}_{i}^{(t)})-f(x^*)\right) \\
			\leq&
			-\Big( \frac{1}{2a}-{L}    -\frac{8L}{3\left(1-\rho({\bf M})\right)^2}\Big){\sum_{\tau=1}^{t}  \lVert  {\bf y}^{(\tau)}-{\bf y}^{(\tau-1)} \rVert^2}\\
			&+\frac{n}{2a}\lVert x^* \rVert^2+ \frac{8\pi^2 }{3L\left(1-(\rho({\bf M}))^2\right)}.
		\end{split}
	\end{equation}
	Upon using the condition in \eqref{step size_FDDA_unconstrained}, we arrive at \eqref{rate_FDDA_unconstrained} as desired.			
\end{proof}

\section{Proof of Theorem \ref{ADDA_thm}}
\subsection{Preliminaries}
For Algorithm \ref{ADDA}, we define 
\begin{equation*}
	{\bf u}^{(t)}=\begin{bmatrix}
		u_{1}^{(t)} \\ \vdots \\ u_{n}^{(t)}
	\end{bmatrix}, \quad
	{\bf v}^{(t)}=\begin{bmatrix}
		v_{1}^{(t)} \\ \vdots \\ v_{n}^{(t)}
	\end{bmatrix},\quad {\bf w}^{(t)}=\begin{bmatrix}
		w_{1}^{(t)} \\ \vdots \\ w_{n}^{(t)}
	\end{bmatrix},
\end{equation*}

\begin{equation*}
	{\bf q}^{(t)}=\begin{bmatrix}
		q_{1}^{(t)} \\ \vdots \\ q_{n}^{(t)}
	\end{bmatrix}, \,\,\hat{\bf \nabla}^{(t)}=\begin{bmatrix}
		{\nabla} f_1(u_{1}^{(t)})\\ \vdots \\ 	\nabla f_n(u_{n}^{(t)})
	\end{bmatrix}, 
\end{equation*}

\begin{equation*}
	\overline{u}^{(t)}=\frac{1}{n}\sum_{i=1}^{n}u_i^{(t)}, \, \overline{v}^{(t)}=\frac{1}{n}\sum_{i=1}^{n}v_i^{(t)}, \, \overline{{w}}^{(t)}=\frac{1}{n}\sum_{i=1}^{n}{  {w}}_{i}^{(t)},
\end{equation*}
\begin{equation*}
	\overline{\hat{g}}^{(t)}=\frac{1}{n}\sum_{i=1}^{n}\nabla f_i(u_{i}^{(t)}),\,
	\overline{{q}}_{t}=\frac{1}{n}\sum_{i=1}^{n}{  {q}}_{i}^{(t)} , \,  \tilde{ {\bf w}}^{(t)}= {\bf w}^{(t)}-{\bf 1}\otimes \overline{w}^{(t)},
\end{equation*}
\begin{equation*}
	\tilde{ {\bf u}}^{(t)}= {\bf u}^{(t)}-{\bf 1}\otimes \overline{u}^{(t)}, \, 	\tilde{ {\bf v}}^{(t)}= {\bf v}^{(t)}-{\bf 1}\otimes \overline{v}^{(t)},\, 	\tilde{ {\bf q}}^{(t)}= {\bf q}^{(t)}-{\bf 1}\otimes \overline{q}^{(t)}.
\end{equation*}
Based on these notations, we present the steps in \eqref{ADDA_iteration} and \eqref{consensus_ADDA} in the following compact form
\begin{subequations}
	\begin{align}
		{\bf u}^{(t)} &= \frac{A_{t-1}}{A_t} \left({\bf P}   {\bf v}^{(t-1)} \right)+ \frac{a_{t}}{A_t} {\bf w}^{(t-1)}, \label{consensus_compact_u}\\
		{\bf v}^{(t)} &= \frac{A_{t-1}}{A_t}\left({\bf P}   {\bf v}^{(t-1)}\right) + \frac{a_{t}}{A_t} {\bf w}^{(t)}, \label{consensus_compact_v}\\
		{\bf q}^{(t)} &= {\bf P}   {\bf q}^{(t-1)} +  \hat{\nabla}^{(t)}-\hat{\nabla}^{(t-1)}, \label{consensus_compact_q}
	\end{align}
\end{subequations}
where ${\bf P} = P\otimes I$. According to \eqref{ADDA_iteration}, we have
\begin{subequations}\label{mean_ADDA}
	\begin{align}
		\overline{u}^{(t)}&=\frac{A_{t-1}}{A_{t}}\overline{v}^{(t-1)}+\frac{a_{t}}{A_{t}}\overline{{w}}^{(t-1)} \label{mean_ADDA_u}, \\
		\overline{v}^{(t)}&=\frac{A_{t-1}}{A_{t}}\overline{v}^{(t-1)}+\frac{a_{t}}{A_{t}}	\overline{{w}}^{(t)}\label{mean_ADDA_v}.
	\end{align}
\end{subequations}



Before proving Theorem \ref{ADDA_thm}, we present Lemma \ref{primal_error} that establishes {upper bounds} for consensus error vectors $\tilde{\bf u}^{(t)}$ and $\tilde{\bf v}^{(t)}$. 

\begin{lemma}\label{primal_error}
	For Algorithm \ref{ADDA}, if Assumptions \ref{lipschitzassumption}, \ref{graphconnected}, and \ref{boundedness} are satisfied, then
	\begin{equation}\label{x2_consensus_error}
		\lVert \tilde{\bf v}^{(t)}  \rVert \leq\frac{a_{t}}{A_{t}}C_p, \quad	\lVert  \tilde{\bf u}^{(t)}  \rVert \leq \frac{a_{t}}{A_{t}}C_p
	\end{equation}	
	for all $t\geq 1$, where $C_p=\lceil\frac{3}{1-\beta}\rceil\sqrt{n}G$, and $\beta =\sigma_2(P)$.
\end{lemma}

\begin{proof}[Proof of Lemma \ref{primal_error}]
	Since both $u_{i}^{(t)}, v_{i}^{(t)}, i=1,\cdots,n$, $\overline{  u}^{(t)}$ and $\overline{  v}^{(t)}$ are within the constraint set,	
	we readily have
	\begin{equation*}
		\begin{split}
			\lVert {\bf u}^{(t)}-{\bf 1}\otimes  \overline{u}^{(t)} \rVert   \leq \sqrt{n}G \,\,\,\, \textrm{and}\,\,\,\,
			\lVert {\bf v}^{(t)}-{\bf 1}\otimes  \overline{v}^{(t)} \rVert   \leq \sqrt{n}G
		\end{split}
	\end{equation*}
	by Assumption \ref{boundedness}.
	Upon using
	\begin{equation*}
		\frac{a_t}{A_t}  =\frac{2(t+1)}{t(t+3)}  \geq \frac{1}{t}, \quad \forall t\geq 1
	\end{equation*}
	and the definition of $C_p=\lceil\frac{3}{1-\beta}\rceil\sqrt{n}G$, 
	we have that \eqref{x2_consensus_error} holds for
	$
		1\leq t< \lceil\frac{3}{1-\beta}\rceil.
	$

	When $t\geq  \lceil\frac{3}{1-\beta}\rceil,$ we prove by an induction argument.
	Suppose that
	\eqref{x2_consensus_error} holds for some $t\geq \lceil \frac{3}{1-\beta}\rceil$. Next, we examine the upper bounds for $\lVert\tilde {\bf v}^{(t+1)} \rVert$ and $\lVert\tilde{\bf u}^{(t+1)} \rVert$, respectively.
	
	\emph {i) Upper bound for $\lVert\tilde {\bf v}^{(t+1)} \rVert$}. 
	Using
	\begin{align*}
		&\mathbf{P}{\bf v}^{(t)}-\mathbf{1}\otimes\overline{v}^{(t)} = \left(\left(P-\frac{\mathbf{1}\mathbf{1}^{\mathrm{T}}}{n}\right)\otimes I\right)\tilde{\bf v}^{(t)},
	\end{align*}
	\eqref{consensus_compact_v} and \eqref{mean_ADDA_v},
	we obtain 
	\begin{equation*}
		\begin{split}
			\tilde{\bf v}^{(t+1)} =&\frac{A_t}{A_{t+1}}\left(\left(P-\frac{\mathbf{1}\mathbf{1}^{\mathrm{T}}}{n}\right)\otimes I\right)\tilde{\bf v}^{(t)} + \frac{a_{t+1}}{A_{t+1}}\tilde{\bf w}^{(t+1)}.
		\end{split}
	\end{equation*}
	Evaluating the norm of both sides of the above equality yields
	\begin{equation*}
		\begin{split}
			\lVert\tilde{\bf v}^{t+1} \rVert 
			=& \left\lVert\frac{A_{t}}{A_{t+1}}\left(\left(P-\frac{\mathbf{1}\mathbf{1}^{\mathrm{T}}}{n}\right)\otimes I\right)\tilde{\bf v}^{(t)}+\frac{a_{t+1}}{A_{t+1}}\tilde{\bf w}^{(t+1)}\right\rVert\\
			{\leq}& \beta\lVert \tilde{\bf v}^{(t)} \rVert+\frac{a_{t+1}}{A_{t+1}}\left\lVert \tilde{\bf w}^{(t+1)} \right\rVert\\
			\leq & \frac{a_t}{A_t}\beta C_p + \frac{a_{t+1}}{A_{t+1}}\sqrt{n}C_p ,
		\end{split}
	\end{equation*}
	where the last inequality follows from the hypothesis that $	\lVert \tilde{\bf v}^{(t)}  \rVert \leq{a_{t}}C_p/{A_{t}}$ and  Assumption \ref{boundedness}.
	Since ${a_t}/A_t$ monotonically decreases with $t$, we have
	\begin{equation*}
		\begin{split}
			\lVert\tilde{\bf v}^{t+1} \rVert 
			\leq& \frac{a_{t}}{A_{t}}\left( \beta C_p+\sqrt{n}G\right) 
			\leq  \frac{a_{t}}{A_{t}}C_p\left( \beta+\frac{1}{\lceil\frac{3}{1-\beta}\rceil}\right) ,
		\end{split}
	\end{equation*}
	where the last inequality is due to	 $\sqrt{n}G=\frac{C_p}{\lceil\frac{3}{1-\beta}\rceil}$.
	It then remains to prove 
	\begin{equation}\label{intermediate_key}
		\left(  {\beta }+\frac{1}{\lceil\frac{3}{1-\beta}\rceil}\right)\leq \frac{A_{t}}{a_{t}}\cdot \frac{a_{t+1}}{A_{t+1}}, \quad  \forall t\geq \lceil\frac{3}{1-\beta}\rceil
	\end{equation}
	to obtain the bound for $\lVert\tilde {\bf v}^{(t+1)} \rVert$ as desired.
	To prove \eqref{intermediate_key}, we let
	\begin{equation*}
		t_0= \lceil\frac{3}{1-\beta}\rceil,
	\end{equation*}
	which implies
	\begin{equation*}
		\frac{3}{t_0} \leq 1-\beta.
	\end{equation*}
	Based on the above relation, we further obtain
	\begin{equation}\label{intermediate}
		\beta +\frac{1}{t_0}\leq \frac{t_0-2}{t_0}\leq \frac{t_0+2}{t_0+4}.
	\end{equation}
	This in conjunction with
	\begin{equation*}
		\frac{t(t+3)}{(t+1)(t+1)} \geq 1, \quad \forall t\geq 1
	\end{equation*}
	and the definitions of $a_{t}$ and $A_{t}$ yields
	\begin{equation*}
		\beta +\frac{1}{t_0}{\leq} \frac{t_0+2}{t_0+4}\cdot\frac{t_0(t_0+3)}{(t_0+1)(t_0+1)}=\frac{A_{t_0}}{a_{t_0}}\cdot \frac{a_{t_0+1}}{A_{t_0+1}}.
	\end{equation*}   
	Since
	$
	{A_{t}a_{t+1}}/{(a_{t}A_{t+1})}
	$ monotonically increases with $t$, we have \eqref{intermediate_key} satisfied.

	\emph {ii) Upper bound for $\lVert\tilde{\bf u}^{(t+1)} \rVert$}. 
	Using the same arguments as above, we have
	\begin{equation*}
		\begin{split}
			\lVert\tilde{\bf u}^{t+1} \rVert 
			= & \left\lVert\frac{A_{t}}{A_{t+1}}\left(\left(P-\frac{\mathbf{1}\mathbf{1}^{\mathrm{T}}}{n}\right)\otimes I\right)\tilde{\bf v}^{(t)}+\frac{a_{t+1}}{A_{t+1}}\tilde{\bf w}^{(t)}\right\rVert\\
			{\leq}& \beta\lVert \tilde{\bf v}^{(t)} \rVert+\frac{a_{t+1}}{A_{t+1}}\left\lVert \tilde{\bf w}^{(t)} \right\rVert\\
			\leq & \frac{a_t}{A_t}\beta C_p + \frac{a_{t+1}}{A_{t+1}}\sqrt{n}C_p .
		\end{split}
	\end{equation*}
	By following the same line of reasoning as in the first part, we are able to obtain
	\begin{equation*}
		\lVert  \tilde{\bf u}^{(t+1)}  \rVert \leq \frac{a_{t+1}}{A_{t+1}}C_p.
	\end{equation*}
	
	Summarizing the above bounds, the proof is completed.
\end{proof}


Lemma \ref{dual_error_lemma} proves the upper bound for the consensus vector $\tilde{\bf q}^{(t)}$. 
\begin{lemma}\label{dual_error_lemma}
	Suppose Assumptions \ref{lipschitzassumption}, \ref{graphconnected}, and \ref{boundedness} are satisfied. For Algorithm \ref{ADDA}, we have 
	\begin{equation}\label{conservation_ADDA}
		\overline{q}^{(t)}=\overline{\hat{g}}^{(t)}
	\end{equation}
	and
	\begin{equation}\label{dual_error_ADDA}
		\lVert \tilde{\bf q}^{(t)}\rVert \leq \frac{a_{t}}{A_{t}}C_g
	\end{equation}
	for all $t\geq 1$,
	where $C_g={\lceil\frac{3}{1-\beta}\rceil L\Big( 2\sqrt{n}G+2C_p\Big)}/{(1-\beta)}$, and $\beta =\sigma_2(P)$.
\end{lemma}
\begin{proof}[Proof of Lemma \ref{dual_error_lemma}]
	The proof of  \eqref{conservation_ADDA} directly follows from the proof of Lemma \ref{conservation_property}, and is omitted here for brevity.

	For \eqref{dual_error_ADDA}, we  subtract ${\bf 1}\otimes \overline{q}^{(t)}$ from both sides of \eqref{consensus_compact_q} to get
	\begin{equation}
		\begin{split}
			{\bf q}^{(t)} -{\bf 1}\otimes \overline{q}^{(t)}=& {\bf P}   {\bf q}^{(t-1)} -{\bf 1}\otimes \overline{q}^{(t-1)}\\
			& +  \hat{\nabla}^{(t)}-\hat{\nabla}^{(t-1)} - {\bf 1}\otimes (\overline{q}^{(t)}-\overline{q}^{(t-1)}). \\
		\end{split}
	\end{equation}
	Using the same procedure in \eqref{consensus_transformation} leads to
	\begin{equation}\label{def:tilde_q}
		\lVert  \tilde{\bf q}^{(t)} \rVert  \leq \beta \lVert  \tilde{\bf q}^{(t-1)} \rVert +\lVert \hat{\nabla}^{(t)}-\hat{\nabla}^{(t-1)} - {\bf 1}\otimes (\overline{q}^{(t)}-\overline{q}^{(t-1)})\rVert.
	\end{equation}		
	Since the objective is smooth, we obtain
	\begin{equation*}
		\begin{split}
			&\left\lVert \hat{\nabla}^{(t)}- \hat{\nabla}^{(t-1)}-\mathbf{1}\otimes \left(\overline{q}^{(t)}-\overline{q}^{(t-1)}\right)\right\rVert\\ 
			=&    \left\lVert \hat{\nabla}^{(t)}- \hat{\nabla}^{(t-1)}-\mathbf{1}\otimes \left(\overline{\hat{g}}^{(t)}-\overline{\hat{g}}^{(t-1)}\right)\right\rVert \\ 
			= & \left\lVert \hat{\nabla}^{(t)}- \hat{\nabla}^{(t-1)}-\left(\frac{\mathbf{1}\mathbf{1}^{\mathrm{T}}}{n}\otimes I\right) \left(\hat{\nabla}^{(t)}- \hat{\nabla}^{(t-1)}\right)\right\rVert \\
			\leq &  \left \lVert\hat{\nabla}^{(t)}- \hat{\nabla}^{(t-1)} \right\rVert	\leq L\left\lVert {\bf u}^{(t)}-{\bf u}^{(t-1)} \right\rVert.
		\end{split}
	\end{equation*} 
	To bound $\left\lVert {\bf u}^{(t)}-{\bf u}^{(t-1)} \right\rVert$, we consider
	\begin{equation*}
		\begin{split}
			& {\bf u}^{(t)}-{\bf u}^{(t-1)} \\
			{	=}& \frac{A_{t-1}}{A_{t}}{\bf P}{\bf v}^{(t-1)}+\frac{a_{t}}{A_{t}}{\bf w}^{(t-1)}-{\bf u}^{(t-1)} \\
			=&\frac{A_{t-1}}{A_{t}}{\bf P}\left({\bf v}^{(t-1)} -{\bf u}^{(t-1)} \right)+\frac{A_{t-1}}{A_{t}}\left({\bf P}-I\otimes I\right) {\bf u}^{(t-1)}\\
			& +\frac{a_{t}}{A_{t}}\left({\bf w}^{(t-1)}-{\bf u}^{(t-1)}\right)
		\end{split}
	\end{equation*}
	where the first equality is due to \eqref{consensus_compact_u}. From \eqref{consensus_compact_u} and \eqref{consensus_compact_v}, we have
$
			{\bf v}^{(t-1)}-{\bf u}^{(t-1)} = \frac{a_{t-1}}{A_{t-1}}\left( {\bf w}^{(t-1)} -{\bf w}^{(t-2)}\right).
$
	In addition, we have
$
		\left({\bf P}-I\otimes I\right) {\bf u}^{(t-1)}= \left({\bf P}-I\otimes I\right) ({\bf u}^{(t-1)} - {\bf 1} \otimes \overline{u}^{(t-1)} ).
$
	Therefore, it holds that
	\begin{equation}\label{def:suc_u}
		\begin{split}
			& \left\lVert {\bf u}^{(t)}-{\bf u}^{(t-1)} \right\rVert\\
			{\leq}& \frac{A_{t-1}}{A_{t}}\frac{a_{t-1}}{A_{t-1}}\left\lVert {\bf w}^{(t-1)} -{\bf w}^{(t-2)} \right\rVert+\frac{2A_{t-1}}{A_{t}}\left\lVert\tilde{\bf u}^{(t-1)} \right\rVert\\ &+\frac{a_{t}}{A_{t}}\left\lVert{\bf w}^{(t-1)}-{\bf u}^{(t-1)}\right\rVert\\
			\leq& \frac{a_{t}}{A_{t}}\sqrt{n}G+\frac{2a_{t}}{A_{t}}C_p+\frac{a_{t}}{A_{t}}\sqrt{n}G\\
			=&\frac{a_{t}}{A_{t}}\left( 2\sqrt{n}G+2C_p\right)
		\end{split}
	\end{equation}
	where Lemma \ref{primal_error} and Assumption \ref{boundedness} are used to get the second inequality. By substituting \eqref{def:suc_u} into \eqref{def:tilde_q}, we obtain
	\begin{equation}\label{dual_induction}
		\begin{split}
			\lVert  \tilde{\bf q}^{(t)} \rVert  \leq \beta \lVert  \tilde{\bf q}^{(t-1)} \rVert +\frac{a_t}{A_t}L\left(  2\sqrt{n}G+2C_p \right).
		\end{split}
	\end{equation}
	By initialization, we have $\tilde{\bf q}^{(0)}=0$ and therefore
	\begin{equation*}
		\begin{split}
			\left\lVert{\tilde{\bf q}}^{(t_0)}\right\rVert
			\leq& L\left( 2\sqrt{n}G+2C_p\right)\sum_{\tau=1}^{t_0}\beta ^{t_0-\tau}
			\frac{a_{\tau}}{A_{\tau}}
			\leq  \frac{L\left( 2\sqrt{n}G+2C_p\right)}{1-\beta},
		\end{split}
	\end{equation*}
	implying that \eqref{dual_error_ADDA} is valid for 	$
	1\leq t<\lceil \frac{3}{1-\beta}\rceil$. Next, we prove that \eqref{dual_error_ADDA} also holds for $
	t\geq \lceil\frac{3}{1-\beta}\rceil$ by mathematical induction. Suppose that \eqref{dual_error_ADDA} holds true for some $t\geq \lceil \frac{3}{1-\beta}\rceil$. Using this hypothesis and \eqref{dual_induction}, we obtain
	\begin{equation*}
		\begin{split}
			\lVert  \tilde{\bf q}^{(t+1)} \rVert  \leq & \frac{a_{t}}{A_{t}}\beta C_g +\frac{a_{t+1}}{A_{t+1}}L\left(  2\sqrt{n}G+2C_p \right) \\
			\leq & \frac{a_{t}}{A_{t}}C_g\left(  {\beta }+\frac{1}{\lceil\frac{3}{1-\beta}\rceil}\right).
		\end{split}
	\end{equation*}
	Finally, using the same argument with  \eqref{intermediate_key} and \eqref{intermediate} in the proof of Lemma \ref{primal_error}, we arrive at \eqref{dual_error_ADDA} as desired.
\end{proof}

\subsection{Proof of Theorem \ref{ADDA_thm}}

\begin{proof}[Proof of Theorem \ref{ADDA_thm}]
	Using $A_{\tau-1}=A_{\tau}-a_\tau$, we have
	\begin{equation*}
		\begin{split}
			&A_{t}\Big(f(\overline{v}^{(t)})-f(x^*)\Big) \\
			=& \sum_{\tau=1}^{t}\left( A_{\tau} f(\overline{v}^{(\tau)})-A_{\tau-1}f(\overline{v}^{(\tau-1)})\right) -\sum_{\tau=1}^{t}a_\tau f(x^*)\\
			=& \sum_{\tau=1}^{t}\Big(A_{\tau} \left(f(\overline{v}^{(\tau)})- f(\overline{u}^{(\tau)})\right)+a_{\tau}\left(f(\overline{u}^{(\tau)})-f(x^*)\right)\\
			&\quad \quad \,\, +A_{\tau-1}\left(f(\overline{u}^{(\tau)})-f(\overline{v}^{(\tau-1)})\right) \Big)\\
		\end{split}
	\end{equation*}
	Upon using the convexity of $f$, we obtain
	\begin{equation*}
		\begin{split}
			&A_{t}\left(f(\overline{v}^{(t)})-f(x^*)\right) \\
			\leq &\sum_{\tau=1}^{t}\Big(A_{\tau} \left(f(\overline{v}^{(\tau)})- f(\overline{u}^{(\tau)})\right)+a_{\tau}\left\langle \nabla f(\overline{u}^{(\tau)}), \overline{u}^{(\tau)} -x ^* \right\rangle\\
			&\quad \quad \,\,+A_{\tau-1}\left\langle\nabla f(\overline{u}^{(\tau)}), \overline{u}^{(\tau)}-\overline{v}^{(\tau-1)}\right\rangle \Big)
		\end{split}
	\end{equation*}
	By \eqref{mean_ADDA_v},
	we obtain
	\begin{equation}\label{ADDA_rate_intermediate}
		\begin{split}
			&A_{t}\Big(f(\overline{v}^{(t)})-f(x^*)\Big) \\
			\leq& \sum_{\tau=1}^{t}A_\tau\Big( f(\overline{v}^{(\tau)})- f(\overline{u}^{(\tau)})+\left\langle\nabla f(\overline{u}^{(\tau)}), \overline{u}^{(\tau)}-\overline{v}^{(\tau)}\right\rangle \Big)\\
			&+\sum_{\tau=1}^{t}a_{\tau}\left\langle \nabla f(\overline{u}^{(\tau)}), \overline{w}^{(\tau)} -x^*  \right\rangle \\
			{\leq} & \sum_{\tau=1}^{t}\frac{A_\tau L}{2}\underbrace{\lVert \overline{u}^{(\tau)}-\overline{v}^{(\tau)} \rVert^2}_{(\uppercase\expandafter{\romannumeral1})}+\frac{1}{n}\sum_{i=1}^{n}\underbrace{\sum_{\tau=1}^{t}a_{\tau}\left\langle q_{i}^{(\tau)}, {w}_{i}^{(\tau)} -x^*  \right\rangle}_{(\uppercase\expandafter{\romannumeral2})}\\
			& +\frac{1}{n}\sum_{\tau=1}^{t}\underbrace{\sum_{i=1}^{n}a_{\tau}\left\langle \nabla f(\overline{u}^{(\tau)}) -q_{i}^{(\tau)}, {w}_{i}^{(\tau)} -x^*  \right\rangle}_{(\uppercase\expandafter{\romannumeral3})}
		\end{split}
	\end{equation}
	where the last inequality is due to the smoothness of $f$.
	To bound $(\uppercase\expandafter{\romannumeral1})$, we consider
	\begin{equation}\label{d}
		\begin{split}
			&\lVert \overline{u}^{(\tau)}-\overline{v}^{(\tau)} \rVert^2\\
			=&\frac{1}{n}\sum_{i=1}^{n}\left\lVert \overline{u}^{(\tau)}-{u}^{(\tau)}_{i}+{u}^{(\tau)}_{i}-{v}^{(\tau)}_{i}+{v}^{(\tau)}_{i}-\overline{v}^{(\tau)} \right \rVert^2\\
			\leq & \frac{1}{n}\sum_{i=1}^{n}\Big(3\left\lVert \overline{u}^{(\tau)}-{u}^{(\tau)}_{i} \right\rVert^2+3\left\lVert {u}^{(\tau)}_{i}-{v}^{(\tau)}_{i} \right\rVert^2\\
			&\quad \quad \quad +3\left\lVert {v}^{(\tau)}_{i}-\overline{v}^{(\tau)} \right\rVert^2\Big) \\
			\leq &  \left(\frac{a_{\tau}}{A_{\tau}} \right)^2\frac{ 6C_p^2+{3}\left\lVert {\bf w}^{(\tau)}-{\bf w}^{(\tau-1)} \right\rVert^2}{n}
		\end{split}
	\end{equation}
	where the first inequality follows from
	$
		\lVert x+y+z\rVert^2 \leq	3\lVert x\rVert^2+	3\lVert y\rVert^2+	3\lVert z\rVert^2,
	$
	and the  last inequality is due to Lemma \ref{primal_error} and \eqref{mean_ADDA}.
	For $(\uppercase\expandafter{\romannumeral2})$, by letting $\zeta^{(\tau)}=q_{i}^{(\tau)}$ and $\nu^{(\tau)}={  w}_{i}^{(\tau)}$ in Lemma \ref{inexact_dual_averaging_inter}, we have
	\begin{equation}\label{e}
		\begin{split}
			\sum_{\tau=1}^{t}a_\tau\left\langle  q_{i}^{(\tau)},{w}_{i}^{(\tau)}-x^*\right \rangle \leq d(x^*)-\sum_{\tau=1}^{t}\frac{1}{2}\lVert  {w}_{i}^{(\tau)}-{w}_{i}^{(\tau-1)} \rVert^2.
		\end{split}
	\end{equation}
	To bound $(\uppercase\expandafter{\romannumeral3})$, we use \eqref{conservation_ADDA} to get
	\begin{equation*}
		\begin{split}
			&a_{\tau}\left\langle \nabla f(\overline{u}^{(\tau)})-q_{i}^{(\tau)}, {{w}}_{i}^{(\tau)} -x^*  \right\rangle\\
			\leq & Ga_\tau\left\lVert\nabla f(\overline{u}^{(\tau)})-\overline{\hat{g}}^{(\tau)}+\overline{q}^{(\tau)}-q_{i}^{(\tau)} \right\rVert \\
			\leq & Ga_\tau  \left(\left\lVert\nabla f(\overline{u}^{(\tau)})-\overline{\hat{g}}^{(\tau)} \right\rVert+ \left\lVert \overline{q}^{(\tau)}-q_{i}^{(\tau)} \right\rVert\right).
		\end{split}
	\end{equation*}
	Upon using Lemma \ref{primal_error}, we obtain
	\begin{equation*}
		\begin{split}
			&\left\lVert \nabla f(\overline{u}^{(\tau)})-\overline{\hat{g}}^{(\tau)} \right\rVert \leq  \frac{1}{n}\sum_{i=1}^{n}\left\lVert \nabla f_i(\overline{u}^{(\tau)})-\nabla f_i(u_{i}^{(\tau)})\right\rVert\\
			\leq & \frac{L}{n}\sum_{i=1}^{n}\left\lVert \overline{u}^{(\tau)}-u_{i}^{(\tau)}\right\rVert  \leq  L \sqrt{   \frac{\lVert \tilde{\bf u}\rVert^2}{n}} 
			\leq  \left(\frac{a_\tau}{A_\tau}\right)\frac{LC_p}{\sqrt{n}}.
		\end{split}
	\end{equation*}
	Recall Lemma  \ref{dual_error_lemma} that
	$\left\lVert \overline{q}^{(\tau)}-q_{i}^{(\tau)} \right\rVert\leq {C_ga_{\tau}}/(\sqrt{n}A_\tau)$.  Therefore
	\begin{equation}\label{f}
		\begin{split}
			&	\sum_{i=1}^{n}a_{\tau}\left\langle \nabla f(\overline{u}^{(\tau)})-q_{i}^{(\tau)}, {{w}}_{i}^{(\tau)} -x^*  \right\rangle\\
			&\leq  \left(\frac{a_\tau^2}{A_\tau} \right) {\sqrt{n}G(LC_p+C_g)}.
		\end{split}
	\end{equation}
	Finally, by collectively substituting \eqref{d}, \eqref{e}, and \eqref{f} into \eqref{ADDA_rate_intermediate}, we get
	\begin{equation*}
		\begin{split}
			&A_{t}\left(f(\overline{v}^{(t)})-f(x^*)\right) \\
			{\leq} & \left(\frac{G(LC_p+C_g)}{\sqrt{n}}+\frac{3LC_p^2}{n}\right)\sum_{\tau=1}^{t}\frac{a_{\tau}^2}{A_{\tau}} \\
			&+d(x^*)+\frac{1}{2n}\sum_{\tau=1}^{t}\left(\frac{3La_{\tau}^2}{A_{\tau}}-1\right)\left\lVert  {\bf w}^{(\tau)}-{\bf w}^{(\tau-1)} \right\rVert^2.
		\end{split}
	\end{equation*}
	Based on the condition in \eqref{ADDA_condition} and the  fact that ${a_\tau^2}/{A_\tau}\leq 2a$, we obtain \eqref{ADDA_rate} as desired.
	
	The inequality in \eqref{consensus_error_ADDA} directly follows from Lemma \ref{primal_error}.
\end{proof}


	


%
%

\ifCLASSOPTIONcaptionsoff
  \newpage
\fi


\bibliographystyle{IEEEtran}
\bibliography{barejrnl.bib}

\end{document}